 \documentclass[10pt,reqno]{amsart}
  \usepackage{geometry}
\usepackage{mathtools}
\newcounter{rel}[table]
\renewcommand{\therel}{Rel.\arabic{rel}}
\usepackage{caption} 
\usepackage{amsmath,amsthm, amscd, amssymb, amsfonts, mathrsfs}
\usepackage[all,cmtip]{xy}
\usepackage{color}

\usepackage[T1]{fontenc}
\usepackage{array,multirow}

\usepackage{ulem,tikz}
\usetikzlibrary{patterns}
\usetikzlibrary{arrows}
\usetikzlibrary{shapes.misc, positioning}
\usetikzlibrary{backgrounds}

\pgfdeclarepatternformonly{my crosshatch dots}{\pgfqpoint{-1pt}{-1pt}}{\pgfqpoint{5pt}{5pt}}{\pgfqpoint{6pt}{6pt}}%
{
    \pgfpathcircle{\pgfqpoint{0pt}{0pt}}{.5pt}
    \pgfpathcircle{\pgfqpoint{3pt}{3pt}}{.5pt}
    \pgfusepath{fill}
}

\usepackage{multicol}
\usepackage{enumitem}

\usepackage{hyperref}

\allowdisplaybreaks

\newcommand{\cA}{\mathcal{A}}

\newcommand{\FK}{\mathcal{FK}}

\newcommand{\Sn}{{\mathbb S}}

\newcommand{\B}{{\mathbb B}}

\renewcommand{\_}[1]{_{\left( #1 \right)}}

\newcommand{\ot}{{\otimes}}
\newcommand{\ku}{\Bbbk}

\newcommand{\Z}{{\mathbb Z}}
\newcommand{\N}{{\mathbb N}}

\newcommand{\D}{\mathcal{D}}

\newcommand{\Rep}{\operatorname{Rep}}

\newcommand{\End}{\operatorname{End}}

\newcommand{\Ind}{\operatorname{Ind}}

\newcommand\sgn{\operatorname{sgn}}

\newcommand\Hom{\operatorname{Hom}}
\newcommand\Mat{\operatorname{Mat}}

\newcommand\id{\operatorname{id}}

\newcommand\soc{\operatorname{soc}}

\newcommand\head{\operatorname{top}}
\newcommand\im{\operatorname{Im}}
\newcommand\Ker{\operatorname{Ker}}

\theoremstyle{plain}

\newtheorem{convention}{Convention}

\newtheorem{lema}{Lemma}[section]
\newtheorem{theorem}[lema]{Theorem}
\newtheorem{cor}[lema]{Corollary}

\newtheorem{prop}[lema]{Proposition}

\newtheorem{question-app}{Question}

\theoremstyle{definition}

\theoremstyle{remark}
\newtheorem{obs}[lema]{Remark}

\newcommand{\Ls}{L}

\newcommand{\Le}{\varepsilon}

\newcommand{\Pl}{P_{\Ls}}

\newcommand{\Pe}{P_{\Le}}

\newcommand{\ba}{{\mathbf a}}
\newcommand{\bt}{{\mathbf t}}
\newcommand{\bs}{{\mathbf s}}

\newcommand{\fka}{{\mathfrak a}}
\newcommand{\fkb}{{\mathfrak b}}
\newcommand{\fkc}{{\mathfrak c}}
\newcommand{\fkw}{{\mathfrak w}}

\newcommand{\xsoc}{x_{soc}}
\newcommand{\xtop}{x_{top}}
\newcommand{\xLt}{x_{\Ls,\bt}}
\newcommand{\xet}{x_{\Le,\bt}}

\newcommand{\xij}[1]{x_{(#1)}}

\newcommand{\fij}[1]{f_{#1}}
\newcommand{\aij}[1]{a_{#1}}
\newcommand{\tij}[1]{t_{#1}}
\newcommand{\sij}[1]{s_{#1}}

\newcommand{\ow}{\overline{w}}
\newcommand{\oomega}{\overline{\omega}}

\begin{document}

\title[The Green ring of a family of copointed Hopf algebras]{The Green ring of a family of copointed Hopf algebras}

\author[Cristian Vay]{Cristian Vay}

\address{\newline\noindent Facultad de Matem\'atica, Astronom\'\i a, F\'\i sica y Computaci\'on,
Universidad Nacional de C\'ordoba. CIEM -- CONI\-CET. Medina Allende s/n, 
Ciudad Universitaria 5000 C\'ordoba, Argentina}
\email{cristian.vay@unc.edu.ar}

\thanks{\noindent 2010 {\it Mathematics Subject Classification.} 16T20, 17B37, 18M05}

\thanks{C. V. is partially supported by  CONICET PIP 11220200102916CO, Foncyt PICT 2016-3927 \& 2020-SERIEA-02847 and Secyt (UNC). This work was finished during a research stay at the Université
Clermont Auvergne supported by CIMPA-ICTP fellowships program ``Research in Pairs''.
}

\begin{abstract}

The copointed liftings of the Fomin-Kirillov algebra $\mathcal{FK}_3$ over the algebra of functions on the symmetric group $\mathbb{S}_3$ were classified by Andruskiewitsch and the author. We demonstrate here that those associated to a generic parameter are Morita equivalent to the regular blocks of well-known Hopf algebras: the Drinfeld doubles of the Taft algebras and the small quantum groups $u_{q}(\mathfrak{sl}_2)$. The indecomposable modules over these were classified independently by Chen, Chari--Premet and Suter. Consequently, we obtain the indocomposable modules over the generic liftings of $\mathcal{FK}_3$. We decompose the tensor products between them into the direct sum of indecomposable modules. We then deduce a presentation by generators and relations of the Green ring.

\end{abstract}

\maketitle

\section{Introduction}

The distinctive feature of a Hopf algebra, and source of its greatest applications, is that its representations form a tensor category. Therefore it is natural to investigate the structure of the tensor products between its modules. This information is encoded in the Green ring which was first considered in the context of finite groups by J. A. Green \cite{green}. This is the ring generated by the isomorphism classes of modules with sum and product induced by the direct sum and tensor product of modules. A first problem which arises is to compute the indecomposable summands of the tensor product of two indecomposable modules.

In the present work, we address the above problem for the generic liftings of the Fomin-Kirillov algebra $\FK_3$ over the algebra on functions of the symmetric group $\Sn_3$. This is certain infinite subfamily of the copointed Hopf algebras over $\ku^{\Sn_3}$ classified by Andruskiewitsch and the author in \cite{AV1}. Moreover, we give a presentation by generators and relations of the Green ring.

Similar results are found in the literature, for instance, for: the small quantum groups $u_q(\mathfrak{sl}_2)$ \cite{kondosaito}; the Taft algebras \cite{green-taft}, their generalized versions \cite{cibils, gen-taft, gen-2taft} and their Drinfeld doubles \cite{Ch4,EGST1,EGST2,Ch6}; the (twisted) Drinfeld doubles of finite groups \cite{Witherspoon1,Witherspoon2}; the non-semisimple Hopf algebras of dimension $8$ \cite{wakui}; the basic Hopf algebras of finite representation type \cite{basic}; the Kac-Paljutkin type algebras \cite{kac-palt}; and the Wu--Liu--Ding algebras \cite{yayaI, yayaII, yayaIII}. The major obstacles for dealing with other Hopf algebras are that they usually are of infinite representation type and the explicit computations of actions are too involved, as we can appreciate in some partial results, {\it e.~g.} \cite{agustin,PV2}.

Let $\cA$ be a generic copointed Hopf algebra over $\ku^{\Sn_3}$; we precisely define it in Section \ref{sec:A}. In order to achieve our goals, we first demonstrate that $\cA$ is Morita equivalent to every regular ({\it i.e.} non-simple) block of the Drinfeld double $\D(T_n(\zeta))$ of a Taft algebra and also to those of the small quantum groups $u_{q}(\mathfrak{sl}_2)$ with $q^2=\zeta$ a primitive $n$-th root of unity. As a direct consequence, we obtain the classification of the indecomposable $\cA$-modules since the respective $\D(T_n(\zeta))$-modules were classified in \cite{Ch3}, and independently in \cite{ChaP,suter} for $u_{q}(\mathfrak{sl}_2)$. 

\begin{theorem}\label{teo:indecomposable}
The following list constitutes a complete set of non-isomorphic indecomposable $\cA$-modules.
\begin{enumerate}
\item The simple modules $\Le$ and $\Ls$.
\item The projective cover and injective hull  $\Pe$ and $\Pl$ of $\Le$ and $\Ls$, respectively.
\item The syzygies and cosyzygy modules $\Omega^k(\Le)$ and $\Omega^k(\Ls)$, $k\in\Z\setminus\{0\}$.
\item The $(k,k)$-type modules $M_k(\Le,\bt)$ and $M_k(\Ls,\bt)$, $k\in\N$ and $0\neq\bt\in\mathfrak{A}/\sim$.
\end{enumerate}
In particular, $\cA$ is of tame representation type.
\end{theorem}

We briefly describe the modules on the above list in Figure \ref{figure:socle}. Some of them were studied early in \cite{AV2}. Here, $\Le$ denotes the trivial one-dimensional module given by the antipode. The reader can think of $\mathfrak{A}/\sim$ as the projective line $\mathbb{P}^1$.

\begin{figure}[h]
\begin{tikzpicture}[scale=.95,every node/.style={scale=0.95}]

\node (layers) {
\begin{tikzpicture}
 
\node (Pl) {
\begin{tikzpicture}

\node (PL1) at (0,0) {\dotuline{$\quad L\quad$}};

\node[below = 0.01 of PL1] (PL2) {\vphantom{$\Ls$}\dotuline{$\quad\Le\oplus\Le\quad$}};

\node[below = .01 of PL2] (PL3) {$\Ls$};

\node[above = .01 of PL1] (PL)  {\uline{$\qquad\Pl\qquad$}};
\end{tikzpicture}
};

\node[right = .5 of Pl] (Pe) {
\begin{tikzpicture}

\node (Pe1) at (0,0) {\vphantom{$\Ls$}\dotuline{$\quad\Le\quad$}};

\node[below = .01 of Pe1] (Pe2) {\dotuline{$\quad\Ls\oplus\Ls\quad$}};

\node[below = .01 of Pe2] (Pe3) {\vphantom{$\Ls$}$\Le$};

\node[above = .01 of Pe1] (PE)  {\uline{$\qquad\Pe\qquad$}};
\end{tikzpicture}
};

\node[below = .1 of Pl] (Oeo) {
\begin{tikzpicture}

\node (Pe1) at (0,0) {\dotuline{\vphantom{$\Ls$}$\quad\mbox{\footnotesize{2(k+1)}}\Ls\quad$}};

\node[below = .01 of Pe1] (Pe2) {$\vphantom{\Ls}\mbox{\footnotesize{(2k+1)}}\Le$};

\node[above = .01 of Pe1] (PE)  {\uline{$\quad\Omega^{2k+1}(\Le)\quad$}};
\end{tikzpicture}
};

\node[right = .1 of Oeo] (Oee) {
\begin{tikzpicture}

\node (Pe1) at (0,0) {\dotuline{\vphantom{$\Ls$}$\,\mbox{\footnotesize{(2k+1)}}\Le\,$}};

\node[below = .01 of Pe1] (Pe2) {$\quad\mbox{\footnotesize{2k}}\Ls\quad$};

\node[above = .01 of Pe1] (PE)  {\uline{$\quad\Omega^{2k}(\Le)\quad$}};
\end{tikzpicture}
};

\node[right = .1 of Oee] (OLo) {
\begin{tikzpicture}

\node (Pe1) at (0,0) {\dotuline{\vphantom{$\Ls$}$\quad\mbox{\footnotesize{2(k+1)}}\Le\quad$}};

\node[below = .01 of Pe1] (Pe2) {$\mbox{\footnotesize{(2k+1)}}\Ls$};

\node[above = .01 of Pe1] (PE)  {\uline{$\quad\Omega^{2k+1}(\Ls)\quad$}};
\end{tikzpicture}
};

\node[right = .1 of OLo] (OLe) {
\begin{tikzpicture}

\node (Pe1) at (0,0) {\dotuline{\vphantom{$\Ls$}$\quad\mbox{\footnotesize{(2k+1)}}\Ls\quad$}};

\node[below = .01 of Pe1] (Pe2) {\vphantom{$\Ls$}$\mbox{\footnotesize{2k}}\Le$};

\node[above = .01 of Pe1] (PE)  {\uline{$\quad\Omega^{2k}(\Ls)\quad$}};
\end{tikzpicture}
};

\node[right = 1 of Pe] (Mke) {
\begin{tikzpicture}

\node (Pe1) at (0,0) {\dotuline{\vphantom{$\Ls$}$\quad\mbox{\footnotesize{k}}\Ls\quad$}};

\node[below = .01 of Pe1] (Pe2) {\vphantom{$\Ls$}$\mbox{\footnotesize{k}}\Le$};

\node[above = .01 of Pe1] (PE)  {\uline{$\quad M_k(\Le,\bt)\quad$}};

\node[below = .01 of PL2] (PL3) {\vphantom{$\Ls$}};

\end{tikzpicture}
};

\end{tikzpicture}
};

\node[right = .25 of layers] (iso) {
\begin{tikzpicture}

\node (Ls) {$\,\,\,\,\Ls\simeq\Ls^*$};

\node[below = .01 of Ls] (Le) {$\,\,\,\,\Le\simeq\Le^*$};

\node[below = .01 of Le] (PL) {$\,\,\Pl\simeq\Pl^*$};

\node[below = .01 of PL] (Pe) {$\,\,\,\Pe\simeq\Pe^*$};

\node[below = .01 of Pe] (Pe1) {$\Omega^{-k}(\Le)\simeq\Omega^{k}(\Le)^*$};

\node[below = .01 of Pe1] (Pe2) {$\Omega^{-k}(\Ls)\simeq\Omega^{k}(\Ls)^*$};

\node[below = .01 of Pe2] (Pe3) {$\,\,M_k(\Ls,\bt)\simeq M_k(\Le,\bt)^*$};

\end{tikzpicture}
};

\end{tikzpicture}
\caption{Loewy layers of some indecomposable modules; $kM$ denotes the direct sum of $k\in\N$  copies of $M$. The remaining ones can be deduced using the isomorphisms on the right-hand side.}
\label{figure:socle}
\end{figure}

Next, we calculate the indecomposable summands of every tensor product between indecomposable modules imitating the strategy of \cite{Ch4}. This is a case-by-case analysis. We employ an inductive argument when we tensoring with modules of the infinite families. We also need to perform some explicit computations which we carry out in GAP \cite{GAP4} as we explain in Appendix \ref{sec:GAP}. These are very long and tedious calculations to do by hand; in the worst case we have to compute actions over the tensor product of two modules of dimension $18$. Once we have these decompositions, we are ready to present the Green ring of $\cA$ by generators and relations using standard arguments. Explicitly, we demonstrate the following.

\begin{theorem}\label{teo:green ring}
The Green ring of $\cA$ is isomorphic to the commutative $\Z$-algebra generated by the elements $\lambda$, $\rho$, $\omega$, $\oomega$ and $\mu_{k,\bt}$, for all $k\in\N$ and
$\bt\in\mathfrak{A}/\sim$, subject to the relations in Table \ref{relaciones}. Moreover, the set 
\begin{align*}
\mathcal{B}=
\left\{\lambda^i,\lambda^i\rho,\lambda^i\omega^s,\lambda^i\oomega^s,\lambda^i\mu_{k,\bt}\mid i\in\{0,1\},s,k\in\N,\bt\in\mathfrak{A}/\sim\right\} 
\end{align*}
is a $\Z$-basis of the Green ring of $\cA$.
\end{theorem}

\begin{table}[h]
\caption{Generators of the Green ring of $\cA$}\label{generadores}
\def\arraystretch{1.6}
\begin{tabular}{lccccr}\hline
$1:=[\Le]$&$\lambda:=[\Ls]$&$\rho:=[\Pl]$&$\omega:=[\Omega(\Le)]$&$\oomega:=[\Omega^{-1}(\Le)]$&
$\mu_{k,\bt}:=[M_k(\Le,\bt)]$
\\
&&&&&{\footnotesize $\forall k\in\N$, 
$\bt\in\mathfrak{A}/\sim$}
\\
\hline
\end{tabular}
\end{table}

We emphasize that our results provide examples of non-quasitriangular Hopf algebras  \cite[Proposition 30]{AV2} with commutative Green ring, and these form an infinite family of non-isomorphic Morita equivalent Hopf algebras with isomorphic Green ring (this does not necessarily imply that the representation categories are monoidally equivalent\footnote{even in the semisimple case, see {\it e.g.} \cite{tambara}.} but we were not able to see whether or not they are). Moreover, the categories of comodules of this family are monoidally equivalent each other \cite[Proposition 29]{AV2} and hence these Hopf algebras are cocycle deformations each other by \cite[Corollary 5.9]{Schauenburg}. However, not all cocycle deformations of these Hopf algebras are Morita equivalent. Indeed, the non-generic liftings of $\FK_3$ over $\ku^{\Sn_3}$ are also cocycle deformations of the generic liftings but their representation categories are quite different, for instance, their blocks have three or six simple modules, cf. \cite{AV2}. We also observe that $\cA$ is a spherical Hopf algebra and the corresponding quotient category is  monoidally equivalent to the category of $C_2\times\Z$-graded finite-dimensional vector spaces, see Corollary \ref{cor:spherical}.

\begin{table}[h]
\caption{Defining relations of the Green ring of $\cA$}\label{relaciones}
\def\arraystretch{1.6}
\begin{tabular}{lllrclr}\hline
\refstepcounter{rel} (\therel)\label{ring:L ot L}
&&&$\lambda^2$&$=$&$1+2\rho$
\\
\refstepcounter{rel} (\therel)
\label{ring:proj ot L}
&&&$\rho^2$&$=$&$2\rho+2\lambda\rho$
\\
\refstepcounter{rel} (\therel)
\label{ring: omega e ot Pl}
&&&$\omega\rho$&$=$&$\rho+2\lambda\rho$
\\
\refstepcounter{rel} (\therel)
\label{ring: omega -e ot Pl}
&&&$\oomega\rho$&$=$&$\rho+2\lambda\rho$
\\
\refstepcounter{rel} (\therel)
\label{ring:omega e ot omega -e}
&&&$\oomega\omega$&$=$&$1+10\rho$
\\
\refstepcounter{rel} (\therel)
\label{ring:mu k t ot Pl}
&&&$\mu_{k,\bt}\rho$&$=$&$k\rho+k\lambda\rho$
\\
\refstepcounter{rel} (\therel)
\label{ring:mu k t ot omega e}
&&&$\mu_{k,\bt}\omega$&$=$&$k\lambda\rho+\mu_{k,\bt}$
&{\footnotesize $\forall\bs,\bt\in\mathfrak{A}$}
\\
\refstepcounter{rel} (\therel)
\label{ring:mu k t ot omega -e}
&&&$\mu_{k,\bt}\oomega$&$=$&$3k\rho+\lambda\mu_{k,\bt}$
&{\footnotesize $\bs\not\sim\bt$}
\\
\refstepcounter{rel} (\therel)
\label{ring:mu k t ot mu j s}
&&&$\mu_{k,\bt}\mu_{j,\bs}$&$=$&$-2jk\rho+jk\lambda\rho$
&{\footnotesize $\forall j,k\in\N$}
\\
\refstepcounter{rel} (\therel)\label{ring:mu k t ot mu j t}
&&&
$\mu_{k,\bt}\mu_{j,\bt}$&$=$&$-2(j-1)k\rho+(j-1)k\lambda\rho+\mu_{k,\bt}+\lambda\mu_{k,\bt}$
&{\footnotesize $k\leq j$}
\\
\hline
\end{tabular}
\end{table}


We conclude this introduction by comparing our work with Chens's paper \cite{Ch4} which describes the Green ring of $\D(T_{2}(-1))$. This Drinfeld double has two singular ({\it i.e.} simple) blocks and only one which is regular. The latter turns out to be the representation category of a book Hopf algebra $\mathbf{h}$ \cite{AS}. Therefore $\mathbf{h}$ is Morita equivalent to $\cA$ and its indecomposable modules are classified as in Theorem \ref{teo:indecomposable}. The two simple modules of $\mathbf{h}$ are one-dimensional, say $\tilde\Ls$ and the trivial one $\tilde\Le$. Then $\tilde\Ls\ot\tilde\Ls\simeq\tilde\Le$ and tensoring with $\tilde\Ls$ induces an involution on the family of indecomposable modules. Instead, our non-trivial simple module $\Ls$ is of dimension $5$ and $\Ls\ot\Ls\simeq\Le\oplus4\Pl$. This projective summand is then propagated through the remaining tensor products. This makes the computations more difficult than in \cite{Ch4}. It is like $\Rep\cA$ is a deformation of $\Rep\mathbf{h}$ over the projective modules. It would be interesting to develop a method to construct new tensor categories by deforming a known one as it occurs in the present situation.

\

The article is organized as follows. In Section \ref{sec:A} we introduce the generic copointed Hopf algebras over $\ku^{\Sn_3}$. In Section \ref{sec:repA}, we prove that they are Morita equivalent to every regular block of $\D(T_n(\zeta))$ and $u_{q}(\mathfrak{sl}_2)$, and describe their indecomposable modules in detail. The bulk of our work is in Section \ref{sec:greenA} where we address the tensor products of every pair of indecomposable modules and prove Theorem \ref{teo:green ring}. In Appendix \ref{sec:GAP}, we explain how we use GAP in our calculations.

\subsection*{Acknowledgments}

I thank Nicol\'as Andruskiewitsch for the stimulating discussions; Corollary \ref{cor:spherical} answers a question asked by him. I am grateful to the referee for his/her interesting comments and suggestions. I thank Martin Mombelli for pointing out the reference \cite{tambara}.

\section{The family of copointed Hopf algebras}\label{sec:A} 

We work over an algebraically closed field $\ku$ of characteristic zero. Let $\ku^{\Sn_3}$ be the algebra of function on the symmetric group $\Sn_3$. We denote $\delta_g$ the characteristic function of $g\in\Sn_3$ and $e$ the neutral element of $\Sn_3$. We recall that $\ku^{\Sn_3}$ is a Hopf algebra where the comultiplication, antipode and counit of each $\delta_g$ are
\begin{align*}
\Delta(\delta_g)=\sum_{h\in\Sn_3}\delta_h\ot\delta_{h^{-1}g}\quad S(\delta_g)=\delta_{g^{-1}}\quad\mbox{and}\quad\varepsilon(\delta_g)=\delta_g(e).\end{align*}

We define 
\begin{align}\label{eq:set A}
\mathfrak{A}=\left\{(\aij{(23)},\aij{(12)},\aij{(13)})\in\ku^3\mid \aij{(23)}+\aij{(12)}+\aij{(13)}=0\right\}. 
\end{align}
We will consider in $\mathfrak{A}$ the equivalence relation $\bt\sim\bs\Leftrightarrow\bt=\lambda\bs$ for some $\lambda\in\ku^\times$.

We fix $\ba\in\mathfrak{A}$ and for each transposition $(ij)$, we  set
\begin{align*}
\fij{ij}=\sum_{g\in\Sn_3}(\aij{(ij)}-\aij{g^{-1}(ij)g})\delta_g\in\ku^{\Sn_3}.
\end{align*}
The Hopf algebra $\cA_{[{\bf a}]}$ defined in \cite[Definition 3.4]{AV1} is generated by $\xij{12}$, $\xij{23}$, $\xij{13}$ and $\delta_g$, for all $g\in\Sn_3$, subject to the relations
\begin{align*}
\xij{ij}^2=\fij{ij},&\quad \delta_g\xij{ij}=\xij{ij}\delta_{(ij)g},\quad\delta_g\delta_h=\delta_g(h)\delta_g,\\
\xij{23}&\xij{12}+\xij{13}\xij{23}+\xij{12}\xij{13}=0,\\
\xij{12}&\xij{23}+\xij{23}\xij{13}+\xij{13}\xij{12}=0
\end{align*}
for all $(ij),g,h\in\Sn_3$.  The comultiplication, the antipode and the counit of each generator $\xij{ij}$ are
\begin{align*}
\Delta(\xij{ij})&=\xij{ij}\ot1 +\sum_{h\in\Sn_3}\sgn(h)\,\delta_h\ot x_{h^{-1}(ij)h},\\
S(\xij{ij})&=-\sum_{h\in\Sn_3}\sgn(h)\, x_{h^{-1}(ij)h}\,\delta_{h^{-1}(ij)}\quad\mbox{and}\quad\varepsilon(\xij{ij})=0.
\end{align*}
The elements $\delta_g$, $g\in\Sn_3$, generate a Hopf subalgebra isomorphic to $\ku^{\Sn_3}.$

The dimension of $\cA_{[{\bf a}]}$ is $72$ and the elements $x\delta_g$, $g\in\Sn_3$ and $x\in\B$, with
\begin{align*}
\B:=
\left\{
\begin{matrix}
1, &\xij{13}, &\xij{13}\xij{12}, &\xij{13}\xij{12}\xij{13}, &\xij{13}\xij{12}\xij{23}\xij{12},\\
   &\xij{23}, &\xij{12}\xij{13}, &\xij{12}\xij{23}\xij{12},\\
   &\xij{12}, &\xij{23}\xij{12}, &\xij{13}\xij{12}\xij{23},\\
   &          &\xij{12}\xij{23}
\end{matrix}
\right\},
\end{align*}
form a basis. 

The following elements will play a distinguished role
\begin{align}\label{eq:xtop}
\xtop:=&\xij{13}\xij{12}\xij{23}\xij{12},\\
\label{eq:xsoc}
\xsoc:=&(-1-2a)(1-a)-\xtop;\\ 
\label{eq:xLt}
\xLt:=&-t_{(12)}\xij{13}\xij{12}+t_{(13)}\xij{12}\xij{23};\\
\label{eq:xet}
\xet:=&-t_{(12)}\xij{12}\xij{13}+t_{(13)}\xij{23}\xij{12};
\end{align}
where $\bt=(\tij{(23)},\tij{(12)},\tij{(13)})\in\mathfrak{A}$.

It holds that $\cA_{[\ba]}\simeq\cA_{[\mathbf{b}]}$ if and only if $\mathbf{b}=\lambda(\aij{\theta(23)},\aij{\theta(12)},\aij{\theta(13)})$ for a permutation $\theta$ of the transpositions and a non-zero scalar $\lambda$, and this gives a classification of the copointed Hopf algebras over $\ku^{\Sn_3}$, see \cite[Theorem 3.5]{AV1}. The representation theory of $\cA_{[{\bf a}]}$ depends on the number of scalars $\aij{(ij)}$ which are equal, cf. \cite{AV2}. In the present work, we will study the generic case, {\it i.~e.} when the three scalars are different. Thus, without loss of generality, we will adopt the following convention.

\begin{convention}\label{convention}
We fix $a\in\ku\setminus\{1,-\frac{1}{2},-2\}$ and set $\cA:=\cA_{[(1,a,-1-a)]}$. From now on, by module we mean left $\cA$-module.
\end{convention}

\section{The category of \texorpdfstring{$\cA$}{A}-modules}\label{sec:repA}

Here, we first recall from \cite{AV2} the simple and projective modules. Using them we prove the Morita equivalence announced in the introduction. Next, we describe in detail the remaining indecomposable modules.

\subsection{\texorpdfstring{$\Sn_3$}{S3}-degree}

We will use the fact that any module $M$ is by restriction a $\ku^{\Sn_3}$-module, or equivalently, a $\Sn_3$-graded module. Explicitly, the homogeneous component $M[g]$ of degree $g\in\Sn_3$ is the subspace of $M$ spanned by the elements $m\in M$ such that $\delta_h\cdot m=\delta_h(g)\,m$. We point out that $(M\ot N)[g]=\oplus_{h\in\Sn_3}M[h]\ot N[h^{-1}g]$. We denote $\ku_g$ the one-dimensional $\ku^{\Sn_3}$-module concentrated in degree $g\in\Sn_3$.

\subsection{Simple modules}

There are only two simple modules \cite[Theorem 1]{AV2}. By abuse of notation, we denote $\Le$ the simple module determined by the counit.

We denote $L$ the non-trivial simple module. It is five-dimensional with basis $\{v_g\mid e\neq g\in\Sn_3\}$. The action of $\cA$ is determined by
\begin{align}\label{eq:action of A}
v_g\in\Ls[g]\quad\mbox{and}\quad
\xij{ij}\cdot v_g=\begin{cases}
    v_{(ij)g}&\mbox{if $\sgn(g)=1$},\\               
    \fij{ij}(g)\,v_{(ij)g}&\mbox{if $\sgn(g)=-1$.}
    \end{cases}
\end{align}

Clearly, the simple modules are self-dual:
\begin{align*}
\Le^*\simeq\Le\quad\mbox{and}\quad\Ls^*\simeq\Ls. 
\end{align*}

\subsection{Projective modules}\label{subsec:projective}

Given $g\in\Sn_3$, we consider the induced module
\begin{align*}
M_g:=\Ind_{\ku^{\Sn_3}}^{\cA}(\ku_g)=\cA\ot_{\ku^{\Sn_3}}\ku_g.
\end{align*}
Equivalently, $M_g$ is the ideal $\cA\delta_g$. Thus, $\{x\delta_g\mid x\in\B\}$ is a basis, $\dim M_g=12$ and any morphism $f:M_g\longrightarrow N$ is determined by its value on $\delta_g$.

We can compute easily the $\Sn_3$-degree of these basis elements. In fact, using the commutating relations we see that
$
\xij{i_1j_1}\cdots\xij{i_lj_l}\delta_g\in M_{g}[(i_1j_1)\cdots(i_lj_l)g] 
$. It follows that $\dim(M_g[h])=2$ for all $h\in\Sn_3$.

By \cite[Lemma 7]{AV2}, $M_g\simeq M_h$ if $g\neq e\neq h$.
 Moreover, by \cite[Theorem 1]{AV2},
\begin{align*}
\Pe:=M_e\quad\mbox{and}\quad\Pl:=M_{(132)} 
\end{align*}
are the projective covers and injective hulls of $\Le$ and $\Ls$, respectively. In particular,
\begin{align*}
\soc(\Pe)&=\cA\cdot(\xtop\delta_e)\simeq\Le\simeq\head(\Pe)\quad\mbox{and}\\
\soc(\Pl)&=\cA\cdot(\xsoc\delta_{\Ls})\simeq\Ls\simeq\head(\Pl),\quad\mbox{where $\delta_\Ls:=\delta_{(132)}$}.
\end{align*}
The generators of the socles were given in
\cite[Lemmas 10 and 13]{AV2}. It follows that the projective modules are self-dual:
\begin{align*}
\Pe^*\simeq\Pe\quad\mbox{and}\quad\Pl^*\simeq\Pl. 
\end{align*}

The following lemma will be useful.

\begin{lema}\label{le:proje direct summand}
Let $M$ be a module and $m\in M[g]$.
\begin{enumerate}
 \item If $g=(132)$ and  $\xsoc\cdot m\neq0$, then $\cA\cdot m\simeq\Pl$ is a direct summand of $M$.
 \item If $g=e$ and  $\xtop\cdot m\neq0$, then $\cA\cdot m\simeq\Pe$ is a direct summand of $M$.
\end{enumerate}

\end{lema}

\begin{proof}
Since $\Pl=M_{(132)}$ is induced from $\ku_{(132)}$, there exists a morphism $F:\Pl\rightarrow M$ such that $F(\delta_{\Ls})=m$. As $F(\xsoc\delta_{\Ls})\neq0$, $F$ is not zero in $\soc(\Pl)$ which is simple. Then $F$ is a monomorphism and (1) follows because $\Pl$ is injective. The proof of (2) is similar.
\end{proof}

\subsection{Morita equivalence}\label{subsec:morita}

Let $\D(T_n(\zeta))$ be the Drinfeld double of a Taft algebra $T_n(\zeta)$ where $\zeta$ is a primitive root of unity of order $n\geq2$.  Chen  showed that the blocks of $\D(T_n(\zeta))$ are arranged in two Morita equivalence classes: the class of regular blocks and the class of singular ones \cite[Proposition 3.1]{Ch3}. Moreover, each regular block is Morita equivalent to any regular block of $\D(T_m(\zeta'))$ for any $m\geq2$ \cite[Proposition 3.3]{Ch3}. The blocks of the Frobenius-Lusztig kernel $u_q(\mathfrak{sl}_2)$, with $q^2=\zeta$, satisfy the same properties, see for instance \cite[Section 5]{suter}.  These facts are proved in {\it loc.~cit.} by computing the basic algebra of each block. More precisely, the basic algebra corresponding to the regular blocks is the following.

\begin{theorem}[\cite{Ch3,suter}]
The regular blocks of $\D(T_n(\zeta))$ and $u_q(\mathfrak{sl}_2)$ are Morita equivalent to the algebra $B$ generated by $e_i$, $u_i$, $w_i$, with $i\in\{1,2\}$, subject to the relations
\begin{align*}
e_ie_j=\delta_{i,j}e_i,\quad u_1w_1&=u_2w_2,\quad w_1u_1=w_2u_2,\\ e_2u_i=u_ie_1=u_i,&\quad w_je_2=e_1w_j=w_j,\\
u_i^2=w_j^2=u_iu_j=w_iw_j=u_iw_j=&w_ju_i=u_ie_2=e_1u_i=e_2w_j=w_je_1=0,
\end{align*}
for all $i,j\in\{1,2\}$ with $i\neq j$. It holds that $\dim B=8$.
\qed
\end{theorem}

The case $n=2$ has a peculiarity. Explicitly, $\mathcal{D}(T_2(-1))$ and $u_q(\mathfrak{sl}_2)$ has only one regular block and it is closed under tensor products.  This is a consequence of $B$ being a quotient of $\mathcal{D}(T_2(-1))$ and $u_q(\mathfrak{sl}_2)$ by a central group-like element as we explain next. Instead, we point out that the regular blocks of $\mathcal{D}(T_n(\zeta))$ and $u_q(\mathfrak{sl}_2)$ for $n>2$ are not closed under tensor products, cf. \cite{Ch2,kondosaito}.

Let $\mathbf{h}$ be the book Hopf algebra \cite{AS} generated by $G$, $X$ and $Y$ subject to the relations $G^2=1$, $X^2=Y^2=0$, $GX=-XG$, $GY=-YG$ and $XY=YX$, with comultiplication
\begin{align*}
\Delta(G)=G\ot G,\quad\Delta(X)=X\ot G+1\ot X,\quad\Delta(X)=Y\ot 1+G\ot Y. 
\end{align*}

It turns out that $B$ is isomorphic to $\mathbf{h}$ as an algebra by letting $G=e_1-e_2$, $X=u_1-w_2$ and $Y=w_1-u_2$. Thus, $B$ inherits a Hopf algebra structure and we have the following.

\begin{prop}\label{prop:u-1sl2}
Let $q$ be a primitive root of unity of order $2$ and $C_2$ the cyclic group of order $2$. Then there exist central extensions of Hopf algebras
\begin{align*}
\ku\rightarrow\ku C_2\rightarrow\mathcal{D}(T_2(-1))\rightarrow\mathbf{h}\rightarrow\ku
\quad\mbox{and}\quad
\ku\rightarrow\ku C_2\rightarrow u_q(\mathfrak{sl}_2)\rightarrow\mathbf{h}\rightarrow\ku.
\end{align*}
Moreover, the category of representations of $\mathbf{h}$ is monoidally equivalent to the regular block of both $\mathcal{D}(T_2(-1))$ and $u_q(\mathfrak{sl}_2)$.
\end{prop}

\begin{proof}
We recall that $\mathcal{D}(T_2(-1))$ is generated by $a,b,c,d$ subject to the relations $ba=-ab$, $bd=-db$, $ca=-ac$, $dc=-cd$, $bc=cb$, $a^2=d^2=0$, $b^2=c^2=1$ and $da+ad=1-bc$; where $b$ and $c$ are group-like elements and $a$ and $d$ are skew-primitive elements, see for instance \cite[page 1460]{Ch4}. Then the Hopf subalgebra generated by $bc$ gives the desired extension.

In the case of $u_q(\mathfrak{sl}_2)$, this is generated by $E,K,F$ subject to the relations $K^4=1$, $KE=q^2EK$, $KF=q^{-2}FK$, $K^2=0=F^2$ and $EF-FE=\frac{K-K^{-1}}{q-q^{-1}}$; where $K$ is a group-like element and $E$ and $F$ are skew-primitive elements, see for instance \cite{suter}. Then the desired extension is given by considering the element $K^2$.

In consequence, these extensions give rise to monoidal functors from $\Rep\mathbf{h}$ to $\Rep\mathcal{D}(T_2(-1))$ and to $\Rep u_q(\mathfrak{sl}_2)$, respectively. It is not difficult to realise that these functors are equivalences with the respective regular blocks.
\end{proof}

We can appreciate from the description of its simple and projective modules that $\cA$ has only one block. We next show that the corresponding basic algebra is isomorphic to $B$.

\begin{prop}
$\cA$ is Morita equivalent to $\mathbf{h}$ and to the regular blocks of $\D(T_n(\zeta))$ and $u_q(\mathfrak{sl}_2)$ for any $n\geq2$. In particular, the indecomposable $\cA$-modules are classified by Theorem \ref{teo:indecomposable}.
\end{prop}

\begin{proof}
Let $P=\Pe\oplus\Pl$ and $E:=\End_{\cA}(P)$ be the basic algebra of $\cA$. We denote 
\begin{align*}
\left(
\begin{matrix}
e_1,\,p_1
&
g_1,\,g_2
\\
\noalign{\smallskip}
f_1,\,f_2
&
e_2,\,p_2
\end{matrix}
\right)\in
\left(
\begin{matrix}
\Hom_{\cA}(M_e,M_e)
&
\Hom_{\cA}(M_e,M_{(132)}) 
\\
\noalign{\smallskip}
\Hom_{\cA}(M_{(132)},M_e)
&
\Hom_{\cA}(M_{(132)},M_{(132)})
\end{matrix}
\right)
\end{align*}
the morphisms determined by
\begin{itemize}
\item $e_1=\id_{M_e}$ and $e_{2}=\id_{M_{(132)}}$.
 \smallskip
 \item $f_1(\delta_{\Ls})=\xij{12}\xij{13}\delta_e$ and $f_2(\delta_{\Ls})=\xij{23}\xij{12}\delta_e$.
 \smallskip
 \item $g_1(\delta_e)=\xij{13}\xij{12}\delta_{\Ls}$ and $g_2(\delta_e)=\xij{12}\xij{23}\delta_{\Ls}$.
 \smallskip
 \item $p_1(\delta_e)=\delta_e$ and $p_2(\delta_{\Ls})=\xsoc\delta_{\Ls}$.
\end{itemize}
These form a basis of $E$ because
\begin{align*}
E\simeq&
\left(
\begin{matrix}
\Hom_{\ku^{\Sn_3}}(\ku_e,M_e)
&
\Hom_{\ku^{\Sn_3}}(\ku_e,M_{(132)}) 
\\
\noalign{\smallskip}
\Hom_{\ku^{\Sn_3}}(\ku_{(132)},M_e)
&
\Hom_{\ku^{\Sn_3}}(\ku_{(132)},M_{(132)})
\end{matrix}
\right)
\end{align*}
and these $\Hom$ spaces are of dimension two, cf. \S\ref{subsec:projective}. 

By a computation in GAP, we see that
\begin{align*}
&e_1+e_2=\id_P,\quad e_ie_j=\delta_{i,j} e_i,\quad1\leq i,j\leq 2,\\
&f_ie_2=e_1f_i=f_i,\quad g_ie_1=e_2g_i=g_i,\quad g_if_i=0,\quad f_ig_i=0,\quad1\leq i\leq 2,
\\
&g_2f_1=p_2,\quad g_1f_2=-p_2,\quad
f_2g_1=p_1,\quad f_1g_2=-p_1;
\end{align*}
it is enough to verify these equalities after evaluating in $\delta_e$ and $\delta_L$.

These are the same relations defining the basic algebra of $\mathcal{D}(T_2(-1))$ and hence $E\simeq B$, see \cite[page 2818]{Ch3}; notice that our convention for the composition of morphisms is opposite to that in {\it loc.~cit.} This proves the first part of the proposition. Therefore the indecomposable modules over $\cA$ are given by translating \cite[Theorem 3.12]{Ch3}. This is Theorem \ref{teo:indecomposable}.
\end{proof}

We will describe in detail the non-simple non-projective indecomposable modules listed on Theorem \ref{teo:indecomposable} in the successive subsections. We notice that every indecomposable module $M$ satisfy $\head(M)\simeq M/\soc(M)$.
According to Chen, we say that $M$ is of $(m,n)$-type if $\soc(M)$ is the direct sum of $n$ simple modules and $M/\soc(M)$ is the direct sum of $m$ simple modules. Given a module $M$ and $k\in\N$, we will denote $kM$ the direct sum of $k$ copies of $M$.
 We recall that a Morita equivalence also preserves simple, indecomposable, injective and projective modules, exact sequences, injective and projective resolutions.

\subsection{The syzygy modules}

We recall that the syzygy $\Omega$ is an endofunctor of the stable category of modules.  It is computed by taking the kernels of projective covers. Here, $\Omega^k(\Le)$ and $\Omega^k(\Ls)$ shall denote certain representatives of the syzygy applied iteratively to the simple modules $\Le$ and $\Ls$. They are defined inductively using the minimal projective resolutions in \cite[page 772]{Ch1} as follows.

We set $\Omega^0(\Le):=\Le$. For $k\geq1$, $\Omega^{k+1}(\Le)$ is determined by the exact sequence
\begin{align}\label{eq:Omega Le s+1 s even}
0\longrightarrow&\Omega^{k+1}(\Le)\longrightarrow(k+1)\Pe\longrightarrow\Omega^{k}(\Le) \longrightarrow0\quad\mbox{for $k$ even,}\\
\noalign{\smallskip}
\label{eq:Omega Le s+1 s odd}
0\longrightarrow&\Omega^{k+1}(\Le)\longrightarrow(k+1)\Pl\longrightarrow\Omega^{k}(\Le) \longrightarrow0\quad\mbox{for $k$ odd.}
\end{align}
Equally, $\Omega^0(\Ls):=\Ls$ and, for $k\geq1$, $\Omega^{k+1}(\Ls)$ is determined by the exact sequence
\begin{align}
\label{eq:Omega Ls s+1 s even}
0\longrightarrow&\Omega^{k+1}(\Ls)\longrightarrow(k+1)\Pl\longrightarrow\Omega^{k}(\Ls) \longrightarrow0\quad\mbox{for $k$ even,}\\
\noalign{\smallskip}
\label{eq:Omega Ls s+1 s odd}
0\longrightarrow&\Omega^{k+1}(\Ls)\longrightarrow(k+1)\Pe\longrightarrow\Omega^{k}(\Ls) \longrightarrow0\quad\mbox{for $k$ odd.}
\end{align}
Notice that $\Omega(\Omega^k(\Le))\simeq\Omega^{k+1}(\Le)$ and $\Omega(\Omega^k(\Ls))\simeq\Omega^{k+1}(\Ls)$. These are the unique $(k+1,k)$-type modules. Moreover, the modules $\Omega^k(\Le)$ and $\Omega^k(\Ls)$ are characterized as the unique (up to isomorphism) indecomposable modules fitting in the exact sequences \cite[Theorem 3.14]{Ch1}:
\begin{align}
\label{eq:s odd Omega Le s+1}
0\longrightarrow k\,\Le\longrightarrow&\Omega^k(\Le)\longrightarrow(k+1)\Ls\longrightarrow0\quad\mbox{for $k$ odd,}\\
\noalign{\smallskip}
\label{eq:s even Omega Le s+1}
0\longrightarrow k\,\Ls\longrightarrow&\Omega^k(\Le)\longrightarrow(k+1)\Le\longrightarrow0\quad\mbox{for $k$ even,}\\
\noalign{\smallskip}
\label{eq:s odd Omega Ls s+1}
0\longrightarrow k\,\Ls\longrightarrow&\Omega^k(\Ls)\longrightarrow(k+1)\Le\longrightarrow0\quad\mbox{for $k$ odd,}\\
\noalign{\smallskip}
\label{eq:s even Omega Ls s+1}
0\longrightarrow k\,\Le\longrightarrow&\Omega^k(\Ls)\longrightarrow(k+1)\Ls\longrightarrow0\quad\mbox{for $k$ even.}
\end{align}
We remark that the modules on the left and right hand sides of the above sequences coincide with the socle and the top of the middle terms by \cite[Corollary 3.16]{Ch1}. 

\subsection{The cosyzygy modules}

The cosyzygy $\Omega^{-1}$ is also an endofunctor of the stable category but it is calculated by taking the cokernels of injective hulls. By \cite[Theorem 3.14]{Ch1}, $\Omega^{-k}(\Le)$ and $\Omega^{-k}(\Ls)$ are the unique $(k,k+1)$-type modules. Moreover, by dualizing the exact sequences \eqref{eq:s odd Omega Le s+1}-\eqref{eq:s even Omega Ls s+1}, we have that 
\begin{align*}
\Omega^{-k}(\Le)\simeq\left(\Omega^{k}(\Le)\right)^* \quad\mbox{and}\quad\Omega^{-k}(\Ls)\simeq\left(\Omega^{k}(\Ls)\right)^*
\end{align*}
for all $k\in\N$. We point out that $\Omega^{-k}(\Le)$ and $\Omega^{-k}(\Ls)$ can be constructed using the dual exact sequences of \eqref{eq:Omega Le s+1 s even}-\eqref{eq:Omega Ls s+1 s odd}. 

\

\subsection{The \texorpdfstring{$(k,k)$}{kk}-type indecomposable modules}

Let $k\in\N$ and $0\neq\bt\in\mathfrak{A}$. We choose $\widehat{\bt}\in\mathfrak{A}$ such that $(\widehat{t}_{(12)}, \widehat{t}_{(13)})$ and $(t_{(12)}, t_{(13)})$ are linearly independent. Inspired in \cite[Definition 11]{AV2} and \cite[Lemmas 3.29-30]{Ch4}, we define the module $M_k(\Le,\bt)$ as the vector space with basis $\{w_g^\ell\mid g\in\Sn_3,\,1\leq \ell\leq k\}$ and  action given by
\begin{align}
\label{eq:M k Le bt}
w_g^\ell\in M_k(\Le,\bt)[g]\quad&\mbox{and}\\
\noalign{\smallskip}
\notag
\xij{ij}\cdot w_g^\ell&=
\begin{cases}
0&\mbox{if $g=e$,}\\
\noalign{\smallskip}
w_{(ij)g}^\ell&\mbox{if $g\neq e$ and $\sgn(g)=1$,}\\
\noalign{\smallskip}
\fij{(ij)}(g)\,w_{(ij)g}^\ell&\mbox{if $g\neq (ij)$ and $\sgn(g)=-1$,}\\
\noalign{\smallskip}
\tij{(ij)}\,w_{e}^\ell+\widehat{t}_{(ij)}\,w_{e}^{\ell-1}&\mbox{if $g= (ij)$,}\\
\end{cases}
\end{align}
for all $g\in\Sn_3$ and $1\leq\ell\leq k$. It is an straightforward computation to verify that this definition respects the defining relations of $\cA$.

\begin{prop}\label{prop:M k Le bt}
\
\begin{enumerate}
\item $M_k(\Le,\bt)$ is a $(k,k)$-type indecomposable module fitting in the exact sequence
\begin{align}\label{eq:kLe M kLs}
0\longrightarrow k\Le\longrightarrow &M_k(\Le,\bt) \longrightarrow k\Ls\longrightarrow0
\end{align}
for all $k\in\N$ and $0\neq\bt\in\mathfrak{A}$.
\item Any $(k,k)$-type indecomposable module with $k\Le$ as socle is isomorphic to $M_k(\Le,\bt)$ for some $0\neq\bt\in\mathfrak{A}$.
\smallskip
\item $M_k(\Le,\bt)\simeq M_k(\Le,\widetilde{\bt}\,)$ if and only if $\bt\sim\widetilde{\bt}$.
\smallskip
\item The definition of $M_k(\Le,\bt)$ does not depend on the election of $\widehat{\bt}$.
\smallskip
\item For all $k\in\N$ and $0\neq\bt\in\mathfrak{A}$, there exists an exact sequence
\begin{align}\label{eq:M1 Mk+1 Mk}
0\longrightarrow M_1(\Le,\bt)\longrightarrow M_{k+1}(\Le,\bt)\longrightarrow M_{k}(\Le,\bt)
\longrightarrow 0.
\end{align}
Moreover, any module fitting in such an exact sequence is isomorphic to either $M_{k+1}(\Le,\bt)$ or $M_{1}(\Le,\bt)\oplus M_{k}(\Le,\bt)$
\end{enumerate}
\end{prop}

\begin{proof}
The modules $M_1(\Le,\bt)$ are exactly the modules introduced in \cite[Definition 11]{AV2} and hence, for $k=1$, this proposition is \cite[Lemmas 12 and 21]{AV2}. 

We now consider the case $k>1$. We begin by proving that $M_k(\Le,\bt)$ is of $(k,k)$-type. We can see from the very definition that $k\Le$ is a submodule of $\soc M_k(\Le,\bt)$ and $[M_k(\Le,\bt):\Le]=[M_k(\Le,\bt):\Ls]=k$. Suppose there is a copy of $\Ls$ in the socle. Thus, there would exist $0\neq w\in M_k(\Le,\bt)[(132)]$ such that $\xij{12}\xij{23}\cdot w=0$ by the definition of $\Ls$. Let us say $w=\sum_{\ell=1}^k a_\ell\, w_{(132)}^\ell$ with $a_\ell\in\ku$. Then
\begin{align*}
\xij{12}\xij{23}\cdot w&=\tij{(12)}\sum_{\ell=1}^ka_\ell\,w_e^\ell+\widehat{t}_{(12)}\sum_{\ell=1}^{k-1}a_{\ell+1}\,w_e^\ell\\
&=\tij{(12)}a_kw_e^k+
\sum_{\ell=1}^{k-1}\left(\tij{(12)}a_\ell+\widehat{t}_{(12)}a_{\ell+1}\right)w_e^\ell.
\end{align*}
If $\tij{(12)}\neq0$, we see that $a_k=0$ and then deduce inductively that $a_\ell=0$ for all $1\leq\ell< k$. Hence $w=0$, a contradiction. If $\tij{(12)}=0$, then $\tij{(13)}$ should be non-zero and we get the same conclusion by considering $\xij{13}\xij{12}\cdot w$. Therefore $M_k(\Le,\bt)$ is of $(k,k)$-type.

We continue by showing that $M_k(\Le,\bt)$ is indecomposable by induction on $k$. Let us assume the contrary, that $M_k(\Le,\bt)=N\oplus\overline{N}$ is a direct sum of submodules. We claim that $w_{(132)}^1$ belongs to either $N$ or $\overline{N}$. In fact, let assume that $w_{(132)}^1=n+\overline{n}$. Thus, there are scalars $a_1, ..., a_k$ such that
\begin{align*}
n=\sum_{\ell=1}^ka_\ell\,w_{(132)}^\ell\in N\quad\mbox{and}\quad
\overline{n}=(1-a_1)\,w_{(132)}^1-\sum_{\ell=2}^ka_\ell\,w_{(132)}^\ell\in\overline{N}.
\end{align*}
Hence
\begin{align*}
\xij{12}\xij{23}\cdot n&=\tij{(12)}\sum_{\ell=1}^ka_\ell\,w_e^\ell+\widehat{t}_{(12)}\sum_{\ell=1}^{k-1}a_{\ell+1}\,w_e^\ell\in N\quad\mbox{and}\\
\xij{13}\xij{12}\cdot n&=\tij{(13)}\sum_{\ell=1}^ka_\ell\,w_e^\ell+\widehat{t}_{(13)}\sum_{\ell=1}^{k-1}a_{\ell+1}\,w_e^\ell\in N.\
\end{align*}
If $\sum_{\ell=1}^{k-1}a_{\ell+1}\,w_e^\ell\neq0$, then $\sum_{\ell=1}^ka_\ell\,w_e^\ell$ and $\sum_{\ell=1}^{k-1}a_{\ell+1}\,w_e^\ell$ are linearly independent.
Since $(\widehat{t}_{(12)}, \widehat{t}_{(13)})$ and $(t_{(12)}, t_{(13)})$ are also linearly independent, we deduce that $\xij{12}\xij{23}\cdot n$ and $\xij{13}\xij{12}\cdot n$ are linearly independent as well. In particular, we can infer that $\sum_{\ell=1}^{k-1}a_{\ell+1}\,w_e^\ell\in N$. However, if we follow the same reasoning with $\overline{n}$ instead of $n$, we would also conclude that $\sum_{\ell=1}^{k-1}a_{\ell+1}\,w_e^\ell\in\overline{N}$. As $N\cap\overline{N}=0$, $\sum_{\ell=1}^{k-1}a_{\ell+1}\,w_e^\ell=0$ and hence either $n=w_{(132)}^1$ or $\overline{n}=w_{(132)}^1$ as we claimed.

our claim is proved. 

Without loss of generality we can assume that $w_{(132)}^1\in N$. Hence the submodule generated by $w_{(132)}^1$ is contained in $N$, {\it i.~e.} $w_{g}^1\in N$ for all $g\in\Sn_3$.  In consequence, we have the following isomorphisms
\begin{align*}
M_k(\Le,\bt)/(\cA\cdot w_{(132)}^1)\simeq M_{k-1}(\Le,\bt)\simeq N/(\cA\cdot w_{(132)}^1)\oplus\overline{N}.
\end{align*}
By induction, it should hold either $\overline{N}=0$ or $N=\cA\cdot w_{(132)}^1$. If $\overline{N}=0$, the proof is complete. If  $N=\cA\cdot w_{(132)}^1$, there is $b\in\ku$ such that $w_{(132)}^2-bw_{(132)}^1\in\overline{N}$. Then
\begin{align*}
\xij{12}\xij{23}\cdot\left(w_{(132)}^2-bw_{(132)}^1\right)&=
\tij{(12)}w_e^2+\left(\widehat{t}_{(12)}-\tij{(12)}b\right)w_e^1\in\overline{N}\quad\mbox{and}\\
\xij{13}\xij{12}\cdot\left(w_{(132)}^2-bw_{(132)}^1\right)&=
\tij{(13)}w_e^2+\left(\widehat{t}_{(13)}-\tij{(13)}b\right)w_e^1\in\overline{N}.
\end{align*}
Thus, the same reasoning of the above paragraph allows us to infer that $w_e^1\in\overline{N}$. Therefore $w_e^1\in N\cap\overline{N}$ which contradicts our assumption $M_k(\Le,\bt)= N\oplus\overline{N}$. In conclusion, $M_k(\Le,\bt)$ is indecomposable and this finishes the proof of item $(1)$.

\

In order to prove the remaining items we first observe that the existence of the exact sequence \eqref{eq:M1 Mk+1 Mk} is immediate from the very definition. 
On the other hand, we know by \cite[page 783]{Ch1} that there is a unique (up to isomorphism)  indecomposable module fitting in such an exact sequence. Moreover, \cite[Theorem 4.16]{Ch1} states that any $(k+1,k+1)$-type indecomposable module with socle $(k+1)\Le$ is constructed in this way. Thus, we see by induction that item $(2)$ holds for all $k>1$. Equally, as item $(3)$ holds for $k=1$, it does for $k>1$ as well as item $(4)$ by {\it loc.~cit.}
\end{proof}

We now study the dual module $M_k(\Ls,\bt):=M_k(\Le,\bt)^*$. Let $\{\ow_{g}^\ell\mid g\in\Sn_3,\,1\leq \ell\leq k\}$ be the basis of $M_k(\Ls,\bt)$ such that
\begin{align*}
\langle\ow_{g}^\ell,w_h^l\rangle =(-1)^{\delta_e(g)}\,\delta_g(h^{-1})\,\delta_{\ell,k-l+1}\quad\forall\,g,h\in\Sn_3\,1\leq\ell,l\leq k.
\end{align*}
The action of $\cA$ on this basis is given by
\begin{align}
\label{eq:M k Ls bt}
\ow_g^\ell\in M_k(\Ls,\bt)[g]\quad&\mbox{and}\\
\noalign{\smallskip}
\notag
\xij{ij}\cdot\ow_g^\ell&=
\begin{cases}
\tij{(ij)}\ow_{(ij)}^\ell+\widehat{t}_{(ij)}\ow_{(ij)}^{\ell-1}&\mbox{if $g=e$,}\\
\noalign{\smallskip}
\fij{(ij)}(g)\,\ow_{(ij)g}^\ell&\mbox{if $g\neq e$ and $\sgn(g)=1$,}\\
\noalign{\smallskip}
\ow_{(ij)g}^\ell&\mbox{if $g\neq (ij)$ and $\sgn(g)=-1$,}\\
\noalign{\smallskip}
0&\mbox{if $g= (ij)$,}\\
\end{cases}
\end{align}
for all $g\in\Sn_3$ and $1\leq\ell\leq k$. In fact, we have that
\begin{align*}
\langle
\xij{ij}\cdot\ow_g^\ell,w_h^l
\rangle= 
\langle
\ow_g^\ell,S(\xij{ij})\cdot w_h^l
\rangle=&
\langle
\ow_g^\ell,-\sum_{\sigma\in\Sn_3}\sgn(\sigma)\,x_{\sigma^{-1}(ij)\sigma}\,\delta_{\sigma^{-1}(ij)}\cdot w_h^l
\rangle\\
=&
\langle
\ow_g^\ell,\sgn(h) p_{h,(ij)} w_{h(ij)}^l+\sgn(h) q_{h,(ij)} w_{h(ij)}^{l-1}
\rangle
\end{align*}
for certain scalars $p_{(ij),h}$ and $q_{(ij),h}$ given by \eqref{eq:M k Le bt}. 
This is not zero only for $(ij)g=h^{-1}$ and either $\ell=k-l+1$ or $\ell=k-(l-1)+1$. Thus, \eqref{eq:M k Ls bt} follows by a case-by-case analysis which we leave to the reader.

We notice that $M_1(\Ls,\bt)$ is the module $W_\bt(\ku_e,L)$ of \cite[Definition 14]{AV2} although there is a typo in {\it loc.~cit.}.

\begin{prop}
\
\begin{enumerate}
\item $M_k(\Ls,\bt)$ is a $(k,k)$-type indecomposable module fitting in the exact sequence
\begin{align}\label{eq:kLs M kLe}
0\longrightarrow k\Ls\longrightarrow &M_k(\Ls,\bt) \longrightarrow k\Le\longrightarrow0
\end{align}
for all $k\in\N$ and $0\neq\bt\in\mathfrak{A}$.
\item Any $(k,k)$-type indecomposable module with $k\Ls$ as socle is isomorphic to  $M_k(\Ls,\bt)$ for some $0\neq\bt\in\mathfrak{A}$.
\smallskip
\item $M_k(\Ls,\bt)\simeq M_k(\Ls,\widetilde{\bt}\,)$ if and only if $\bt\sim\widetilde{\bt}$.
\smallskip
\item The definition of $M_k(\Ls,\bt)$ does not depend on the election of $\widehat{\bt}$.
\smallskip
\item For all $k\in\N$ and $0\neq\bt\in\mathfrak{A}$, there exists an exact sequence
\begin{align}\label{eq:M1 Mk+1 Mk Ls}
0\longrightarrow M_1(\Ls,\bt)\longrightarrow M_{k+1}(\Ls,\bt)\longrightarrow M_{k}(\Ls,\bt)
\longrightarrow 0.
\end{align}
Moreover, any module fitting in such an exact sequence is isomorphic to either $M_{k+1}(\Ls,\bt)$ or $M_1(\Ls,\bt)\oplus M_k(\Ls,\bt)$.
\end{enumerate}
\end{prop}

\begin{proof}
It follows from Proposition \ref{prop:M k Le bt} by taking duals. 
\end{proof}

We can infer the projective covers, the injective hulls, the syzygies and the cosyzygies of the $(k,k)$-type indecomposable modules from the next result.

\begin{prop}\label{prop:projective of MLt and Met}
For all $k\in\N$ and $0\ne\bt\in\mathfrak{A}$, we have the exact sequences 
\begin{align}\label{eq:MkL P Mke}
0\longrightarrow M_k(\Ls,\bt)\longrightarrow &k\Pl\longrightarrow M_k(\Le,\bt) \longrightarrow 0\quad\mbox{and}\\
\label{eq:Mke P MkL}
0\longrightarrow M_k(\Le,\bt)\longrightarrow &k\Pe\longrightarrow M_k(\Ls,\bt) \longrightarrow 0.
\end{align}
Moreover, for $k=1$, the image of the inclusions in the above exact sequences  satisfy 
\begin{align*}
M_1(\Ls,\bt)\simeq\cA\cdot(\xLt\delta_{\Ls})\subset\Pl\quad\mbox{and}\quad 
M_1(\Le,\bt)\simeq\cA\cdot(\xet\delta_{e})\subset\Pe
\end{align*}
\end{prop}

\begin{proof}
The first part is in \cite[page 1462]{Ch4}. We next prove the claims for $k=1$.

By a computation in GAP, we see that the morphism $\Pl\longrightarrow M_1(\Le,\bt)$ induced by the assignment $\delta_{\Ls}\mapsto w_{(132)}$ is an epimorphism and the element $\xLt\delta_{\Ls}$ belongs to the kernel. By the proof of \cite[Lemma 15]{AV2}, we know that  $\xLt\delta_{\Ls}$ generates a submodule isomorphic to $M_1(\Ls,\bt)$. This proves $M_1(\Ls,\bt)\simeq\cA\cdot(\xLt\delta_{\Ls})\subset\Pl$

Similarly, one can show the other claim. We only note that the kernel of the epimorphism $\Pe\longrightarrow M_1(\Ls,\bt)$ induced by $\delta_e\mapsto\ow_e$ contains $\xet\delta_{e}$ and this element generates a submodule isomorphic to $M_1(\Le,\bt)$ by the proof of \cite[Lemma 12]{AV2}.
\end{proof}

\section{Tensor products between indecomposable modules}\label{sec:greenA}

In the successive subsections, we decompose the tensor product of every pair of indecomposable modules into the direct sum of indocomposable modules. We then prove our main result over the Green ring of $\cA$.

\subsection{Preliminaries}

We recall some well-known facts which are useful in the calculus of tensor products. We will use them without an explicit mention. 

First, tensoring and dualizing are exact functors. The tensor product with a projective module is projective. See for instance \cite{EGNO}

Second, let $M,N,P,Q$ be modules with $P$ and $Q$ projective, and hence injective, forming the exact sequence
\begin{align*}
0\longrightarrow M\oplus P 
\longrightarrow E
\longrightarrow N\oplus Q
\longrightarrow 0
\end{align*}
Then, using the Krull--Schmidt theorem, it follows that $E\simeq\widetilde{E}\oplus P\oplus Q$ for some module $\widetilde{E}$ fitting in the exact sequence
\begin{align*}
0\longrightarrow M
\longrightarrow\widetilde{E}
\longrightarrow N
\longrightarrow 0.
\end{align*}

\begin{lema}[{\cite[Lemma 3.12]{Ch4}}]\label{le:como Chen}
Let $\{S_1,S_2\}=\{\Le,\Ls\}$ and $M$ be an indecomposable module such that $M/\soc M\simeq k S_1$ for some $k\geq1$. If $f:sP_{S_1}\oplus tP_{S_2}\longrightarrow M$ is an epimorphism with $s,t\geq1$, then $s\geq k$ and $\Ker(f)\simeq\Omega(M)\oplus(s-k)P_{S_1}\oplus tP_{S_2}$.
\end{lema}

\begin{proof}
It is exactly \cite[Lemma 3.12]{Ch4} except for the claim that $s\geq k$ which follows using that $kP_{S_1}$ is the projective cover of $M$. 
\end{proof}

In Appendix \ref{sec:GAP} we explain how we implemente our computations in GAP.

\subsection{Simple tensor simple}\label{subsec: L ot L} We have the following isomorphism
\begin{align}\label{iso:L ot L}
\Ls\ot\Ls\simeq\Le\oplus2\Pl.
\end{align}

Indeed, the projective direct summand is generated by  $n_1:=v_{(13)}\ot v_{(23)}$ and $n_2:=v_{(12)}\ot v_{(13)}$. To prove this, according to Lemma \ref{le:proje direct summand}, we must calculate the action of $\xsoc$ on them. By a computation in GAP, we see that:
\begin{align*}
\xsoc\cdot n_1&=(2a^2-a-1)\left(v_{(23)}\ot v_{(12)}-v_{(12)}\ot v_{(13)}-(a+2)v_{(123)}\ot v_{(123)}\right)\quad\mbox{and}\\
\xsoc\cdot n_2&=(2a^2+a-2)\left(v_{(23)}\ot v_{(12)}-v_{(13)}\ot v_{(23)}-(2a+1)v_{(123)}\ot v_{(123)}\right).
\end{align*}
These elements are non-zero as $a\neq1,-\frac{1}{2},-2$. Then   $N_1:=\cA\cdot n_1$ and $N_2:=\cA\cdot n_2$ are direct summands of $\Ls\ot\Ls$  isomorphic to $\Pl$ by Lemma \ref{le:proje direct summand}. Also, $\xsoc\cdot n_1$ and $\xsoc\cdot n_2$ are clearly linearly independent and hence $N_1\cap N_2=0$ because they generate the socle of $N_1$ and $N_2$. Therefore $\Ls\ot\Ls=N_1\oplus N_2\oplus N\simeq 2\Pl\oplus N$ for some submodule $N$. Since $\dim(\Ls\ot\Ls)=\dim(2\Pl)+1$, $N$ must be isomorphic to $\Le$ and \eqref{iso:L ot L} is proved. We point out that $N$ is spanned by
{\small
\begin{align*}
\frac{1}{2a+1}v_{(23)}\ot v_{(23)}-\frac{1}{a+2} v_{(12)}\ot v_{(12)}&-\frac{1}{a-1}v_{(13)}\ot v_{(13)}-v_{(123)}\ot v_{(132)}+v_{(132)}\ot v_{(123)}.
\end{align*}
}

\subsection{Tensoring by projectives}\label{subsec: ot projective} We will prove that
\begin{align}\label{iso:proj ot L}
\Pl\ot\Ls\simeq4\Pl\oplus\Pe\simeq\Ls\ot\Pl\quad\mbox{and}\quad\Pe\ot\Ls\simeq5\Pl\simeq\Ls\ot\Pe.
\end{align}

We begin by proving the first isomorphism. Let us apply $-\ot\Ls$ to the exact sequence \eqref{eq:s odd Omega Ls s+1} for $k=1$:
\begin{align*}
0\longrightarrow\Ls\ot\Ls\overset{\eqref{iso:L ot L}}{\simeq}\Le\oplus2\Pl\longrightarrow\Omega(L)\ot\Ls\longrightarrow2\Le\ot\Ls\simeq2\Ls\longrightarrow0. 
\end{align*}
Then $\Omega(\Ls)\ot\Ls\simeq 2\Pl\oplus N$ for some module $N$ of dimension $11$. On the other hand, we apply $-\ot\Ls$ to the exact sequence \eqref{eq:Omega Ls s+1 s even} for $k=0$:
\begin{align*}
0\longrightarrow\Omega(L)\ot\Ls\simeq2\Pl\oplus N\longrightarrow\Pl\ot\Ls\longrightarrow\Ls\ot\Ls\overset{\eqref{iso:L ot L}}{\simeq}\Le\oplus2\Pl\longrightarrow0.
\end{align*}
Then $\Pl\ot\Ls\simeq4\Pl\oplus Q$ for some projective module $Q$ such that
\begin{align*}
0\longrightarrow N\longrightarrow Q\longrightarrow\Le\longrightarrow0. 
\end{align*}
Therefore $Q\simeq\Pe$ and the first isomorphism of \eqref{iso:proj ot L} follows.  Notice that the same proof runs for $\Ls\ot\Pl$.

\


The proof of the third isomorphism of \eqref{iso:proj ot L} is similar. We first apply $-\ot\Ls$ to the exact sequence \eqref{eq:s odd Omega Le s+1} for $k=1$:
\begin{align*}
0\longrightarrow\Le\ot\Ls\simeq\Ls\longrightarrow\Omega(\Le)\ot\Ls\simeq2\Ls\ot\Ls\overset{\eqref{iso:L ot L}}{\simeq}2\Le\oplus 4\Pl\longrightarrow0. 
\end{align*}
Then $\Omega(\Le)\ot\Ls\simeq4\Pl\oplus N$ for some module $N$ of dimension $7$. Second, we apply $-\ot\Ls$ to the exact sequence \eqref{eq:Omega Le s+1 s even} for $k=0$:
\begin{align*}
0\longrightarrow\Omega(\Le)\ot\Ls\simeq4\Pl\oplus N\longrightarrow\Pe\ot\Ls\longrightarrow\longrightarrow\Le\ot\Ls\simeq\Ls\longrightarrow0. 
\end{align*}
Then $\Pe\ot\Ls\simeq4\Pl\oplus Q$ for some projective module $Q$ such that
\begin{align*}
0\longrightarrow N\longrightarrow Q\longrightarrow\Ls\longrightarrow0. 
\end{align*}
Therefore $Q\simeq\Pl$ which proves the third isomorphism of  \eqref{iso:proj ot L}.   Notice that the same proof runs for $\Ls\ot\Pe$.

\begin{obs}\label{obs:tensoring by proj}
Let $[M:\Ls]$ and $[M:\Le]$ denote the number of composition factors isomorphic to  $\Ls$ and $\Le$, respectively,  of a module $M$. By induction on the length of $M$, it is easy  to see that
\begin{align*}
M\ot\Pl\simeq\Pl\ot M&\simeq[M:\Ls]\Pe\oplus(4[M:\Ls]+[M:\Le])\Pl\quad\mbox{and}\\
M\ot\Pe\simeq\Pe\ot M&\simeq[M:\Le]\Pe\oplus5[M:\Ls]\Pl.
\end{align*}
\end{obs}

\subsection{Tensor products between syzygy modules}\label{subsec:tensor syzygy} Let $s,t\in\Z_{\geq0}$. Then
\begin{align}
\label{iso:Omega L ot Omega L}
\Omega^s(\Ls)\ot\Omega^t(\Ls)
\simeq\,&\Omega^t(\Ls)\ot\Omega^s(\Ls)\\
\notag\simeq\,&
\Omega^{s+t}(\Le)
\oplus
\begin{cases}
s(3t+2)\Pl&\mbox{for $s$ odd and $t$ even;}\\
\noalign{\smallskip}
st\Pe\oplus2st\Pl&\mbox{for $s$ odd and $t$ odd;}\\
\noalign{\smallskip}
st\Pe\oplus2(s+1)(t+1)\Pl&\mbox{for $s$ even and $t$ even;}\\
\noalign{\smallskip}
(3s+2)t\Pl&\mbox{for $s$ even and $t$ odd;}
\end{cases}
\end{align}
\begin{align}\label{iso:Omega L ot Omega e}
\Omega^s(\Ls)\ot\Omega^t(\Le)
\simeq\,&\Omega^t(\Ls)\ot\Omega^s(\Le)\\
\notag
\simeq\,&\Omega^{s+t}(\Ls)
\oplus
\begin{cases}
st\Pe\oplus2st\Pl&\mbox{for $s$ odd and $t$ even;}\\
\noalign{\smallskip}
s(3t+2)\Pl&\mbox{for $s$ odd and $t$ odd;}\\
\noalign{\smallskip}
(3s+2)t\Pl&\mbox{for $s$ even and $t$ even;}\\
\noalign{\smallskip}
st\Pe\oplus2(t+1)(s+1)\Pl&\mbox{for $s$ even and $t$ odd;}
\end{cases}
\end{align}
\begin{align}\label{iso:Omega e ot Omega e}
\Omega^s(\Le)\ot\Omega^t(\Le)\simeq\,&
\Omega^t(\Le)\ot\Omega^s(\Le)\\
\notag
\simeq\,&\Omega^{s+t}(\Le)
\oplus
\begin{cases}
(3s+2)t\Pl&\mbox{for $s$ odd and $t$ even;}\\
\noalign{\smallskip}
st\Pe\oplus2(s+1)(t+1)\Pl&\mbox{for $s$ odd and $t$ odd;}\\
\noalign{\smallskip}
st\Pe\oplus2st\Pl&\mbox{for $s$ even and $t$ even;}\\
\noalign{\smallskip}
(3t+2)s\Pl&\mbox{for $s$ even and $t$ odd;}
\end{cases}
\end{align}

The demonstrations of these isomorphisms are by induction on $s$ and $t$. We only prove \eqref{iso:Omega L ot Omega L} since the remaining proofs are similar. We will compute $\Omega^s(\Ls)\ot\Omega^t(\Ls)$. We observe that our reasoning will rely on the isomorphisms in \eqref{iso:proj ot L} and therefore $\Omega^s(\Ls)\ot\Omega^t(\Ls)
\simeq\Omega^t(\Ls)\ot\Omega^s(\Ls)$.

For $s=t=0$, this is \eqref{iso:L ot L}. Next, we show the isomorphism for $s+1$ and $t=0$. If $s$ is even, we apply $-\ot\Ls$ to \eqref{eq:Omega Ls s+1 s even} with $k=s$. We obtain the exact sequence:
\begin{align*}
0\longrightarrow\Omega^{s+1}(\Ls)\ot\Ls\longrightarrow&(s+1)\Pl\ot\Ls\overset{\eqref{iso:proj ot L}}{\simeq}4(s+1)\Pl\oplus(s+1)\Pe\longrightarrow\\&
\longrightarrow\Omega^{s}(\Ls)\ot\Ls\overset{\mbox{\footnotesize{IH}}}{\simeq}\Omega^{s}(\Le)\oplus2(s+1)\Pl\longrightarrow0.
\end{align*}
Then, by Lemma \ref{le:como Chen}, $\Omega^{s+1}(L)\ot\Ls\simeq\Omega^{s+1}(\Le)\oplus2(s+1)\Pl$ as claimed. If $s$ is odd, we apply $-\ot\Ls$ to \eqref{eq:Omega Ls s+1 s odd} with $k=s$ and obtain the exact sequence
\begin{align*}
0\rightarrow\Omega^{s+1}(\Ls)\ot\Ls\rightarrow&(s+1)\Pe\ot\Ls\overset{\eqref{iso:proj ot L}}{\simeq}5(s+1)\Pl\rightarrow\Omega^{s}(\Ls)\ot\Ls\overset{\mbox{\footnotesize{IH}}}{\simeq}\Omega^{s}(\Le)\oplus s2\Pl\rightarrow0.
\end{align*}
Again, by Lemma \ref{le:como Chen}, $\Omega^{s+1}(L)\ot\Ls\simeq\Omega^{s+1}(\Le)\oplus2(s+2)\Pl$ as claimed. 

Now, with a similar strategy, we show the isomorphism for $s\in\Z_{\geq0}$ and $t+1$. We apply $\Omega^s(\Ls)\ot-$ to the exact sequences \eqref{eq:Omega Ls s+1 s even} or \eqref{eq:Omega Ls s+1 s odd}, depending on the parity of $k=t$, and use Lemma \ref{le:como Chen}. We must analyse four cases. First, if $s$ is odd and $t$ is even, we have the exact sequence
\begin{align*}
0&\longrightarrow\Omega^s(\Ls)\ot\Omega^{t+1}(\Ls)\longrightarrow\\
&
\longrightarrow\Omega^s(\Ls)\ot(t+1)\Pl\overset{\mbox{\tiny Remark \ref{obs:tensoring by proj} and \eqref{eq:s odd Omega Ls s+1}}}{\simeq}(5s+1)(t+1)\Pl\oplus s(t+1)\Pe\longrightarrow\\
&
\longrightarrow\Omega^{s}(\Ls)\ot\Omega^{t}(\Ls)\overset{\mbox{\footnotesize{IH}}}{\simeq}\Omega^{s+t}(\Le)\oplus s(3t+2)\Pl\longrightarrow0.
\end{align*}
Then $\Omega^{s}(L)\ot\Omega^{t+1}(\Ls)\simeq\Omega^{s+t+1}(\Le)\oplus s(t+1)\Pe\oplus 2s(t+1)\Pl$ as claimed. Second, if $s$ is even  and $t$ is even, we have:
\begin{align*}
0&\longrightarrow\Omega^s(\Ls)\ot\Omega^{t+1}(\Ls)\longrightarrow\\
&
\longrightarrow\Omega^s(\Ls)\ot(t+1)\Pl\overset{\mbox{\tiny Remark \ref{obs:tensoring by proj} and \eqref{eq:s even Omega Ls s+1} }}{\simeq}(5s+4)(t+1)\Pl\oplus (s+1)(t+1)\Pe\longrightarrow\\
&
\longrightarrow\Omega^{s}(\Ls)\ot\Omega^{t}(\Ls)\overset{\mbox{\footnotesize{IH}}}{\simeq}\Omega^{s+t}(\Le)\oplus st\Pe\oplus 2(s+1)(t+1)\Pl\longrightarrow0.
\end{align*}
Then $\Omega^s(\Ls)\ot\Omega^{t+1}(\Ls)\simeq\Omega^{s+t+1}(\Le)\oplus(3s+2)(t+1)\Pl$ as claimed. Third, if $s$ is odd and $t$ is odd, we have:
\begin{align*}
0&\longrightarrow\Omega^s(\Ls)\ot\Omega^{t+1}(\Ls)\longrightarrow\\
&
\longrightarrow\Omega^s(\Ls)\ot(t+1)\Pe\overset{\mbox{\tiny Remark \ref{obs:tensoring by proj} and \eqref{eq:s odd Omega Ls s+1} }}{\simeq}5s(t+1)\Pl\oplus (s+1)(t+1)\Pe\longrightarrow\\
&
\longrightarrow\Omega^{s}(\Ls)\ot\Omega^{t}(\Ls)\overset{\mbox{\footnotesize{IH}}}{\simeq}\Omega^{s+t}(\Le)\oplus st\Pe\oplus 2st\Pl\longrightarrow0.
\end{align*}
Then $\Omega^s(\Ls)\ot\Omega^{t+1}(\Ls)\simeq\Omega^{s+t+1}(\Le)\oplus s(3(t+1)+2)\Pl$ as claimed. Finally, the fourth case which we must consider is with $s$ even and $t$ odd. We have:
\begin{align*}
0&\longrightarrow\Omega^s(\Ls)\ot\Omega^{t+1}(\Ls)\longrightarrow\\
&
\longrightarrow\Omega^s(\Ls)\ot(t+1)\Pe\overset{\mbox{\tiny Remark \ref{obs:tensoring by proj} and \eqref{eq:s even Omega Ls s+1} }}{\simeq}5(s+1)(t+1)\Pl\oplus s(t+1)\Pe\longrightarrow\\
&
\longrightarrow\Omega^{s}(\Ls)\ot\Omega^{t}(\Ls)\overset{\mbox{\footnotesize{IH}}}{\simeq}\Omega^{s+t}(\Le)\oplus  (3s+2)t\Pl\longrightarrow0.
\end{align*}
Then $\Omega^s(\Ls)\ot\Omega^{t+1}(\Ls)\simeq\Omega^{s+t+1}(\Le)\oplus s(t+1)\Pe\oplus 2(s+1)((t+1)+1)\Pl$ as claimed. This conclude the double induction and hence \eqref{iso:Omega L ot Omega L} holds.

\subsection{Tensor products between syzygies and cosyzygy modules} Let $s,t\in\Z_{\geq0}$. The next isomorphisms can be demonstrated by a double induction procedure and following a strategy analogous to that used for showing \eqref{iso:Omega L ot Omega L}.

\

\begin{align}\label{eq:omega L ot omega -L}
\Omega^s(\Ls)\ot\Omega^{-t}(\Ls)&\simeq\Omega^{-t}(\Ls)\ot\Omega^{s}(\Ls)\\
\notag
&\simeq\Omega^{s-t}(\Le)
\oplus
\begin{cases}
(s+1)t\Pe\oplus 2s(t+1)\Pl&\mbox{for $s$ odd $>t$ even;}\\
\noalign{\smallskip}
s(t+1)\Pe\oplus 2s(t+1)\Pl&\mbox{for $s$ odd $<t$ even;}\\
\noalign{\smallskip}
(3s+1)t\Pl&\mbox{for $s$ odd $\geq t$ odd;}\\
\noalign{\smallskip}
s(3t+1)\Pl&\mbox{for $s$ odd $<t$ odd;}\\
\noalign{\smallskip}
(s+1)(3t+2)\Pl&\mbox{for $s$ even $\geq t$ even;}\\
\noalign{\smallskip}
(3s+2)(t+1)\Pl&\mbox{for $s$ even $< t$ even;}\\
\noalign{\smallskip}
(s+1)t\Pe\oplus2(s+1)t\Pl&\mbox{for $s$ even $> t$ odd;}\\
\noalign{\smallskip}
s(t+1)\Pe\oplus2(s+1)t\Pl&\mbox{for $s$ even $< t$ odd;}\\
\end{cases}
\end{align}
\begin{align}\label{eq:omega e ot omega -L}
\Omega^{s}(\Ls)\ot\Omega^{-t}(\Le)&\simeq\Omega^{-t}(\Ls)\ot\Omega^{s}(\Le)\\
\notag
&\simeq
\Omega^{s-t}(\Ls)
\oplus
\begin{cases}
(t+1)(3s+2)\Pl&\mbox{for $t$ odd $>s$ even;}\\
\noalign{\smallskip}
(3t+2)(s+1)\Pl&\mbox{for $t$ odd $<s$ even;}\\
\noalign{\smallskip}
(t+1)s\Pe\oplus2(t+1)s\Pl&\mbox{for $t$ odd $\geq s$ odd;}\\
\noalign{\smallskip}
t(s+1)\Pe\oplus2(t+1)s\Pl&\mbox{for $t$ odd $<s$ odd;}\\
\noalign{\smallskip}
(t+1)s\Pe\oplus2t(s+1)\Pl&\mbox{for $t$ even $\geq s$ even;}\\
\noalign{\smallskip}
t(s+1)s\Pe\oplus(s+1)\Pl&\mbox{for $t$ even $< s$ even;}\\
\noalign{\smallskip}
(3t+1)s\Pl&\mbox{for $t$ even $> s$ odd;}\\
\noalign{\smallskip}
t(3s+1)\Pl&\mbox{for $t$ even $< s$ odd;}\\
\end{cases}
\end{align}
\begin{align}\label{eq:omega e ot omega -e}
\Omega^s(\Le)\ot\Omega^{-t}(\Le)&\simeq
\Omega^{-t}(\Le)\ot\Omega^{s}(\Le)\\
\notag
&\simeq
\Omega^{s-t}(\Le)
\oplus
\begin{cases}
(s+1)t\Pe\oplus2(s+1)t\Pl&\mbox{for $s$ odd $>t$ even;}\\
\noalign{\smallskip}
s(t+1)\Pe\oplus2(s+1)t\Pl&\mbox{for $s$ odd $<t$ even;}\\
\noalign{\smallskip}
(s+1)(3t+2)\Pl&\mbox{for $s$ odd $\geq t$ odd;}\\
\noalign{\smallskip}
(3s+2)(t+1)\Pl&\mbox{for $s$ odd $<t$ odd;}\\
\noalign{\smallskip}
(s+1)(3t+1)\Pl&\mbox{for $s$ even $\geq t$ even;}\\
\noalign{\smallskip}
(3s+1)(t+1)\Pl&\mbox{for $s$ even $< t$ even;}\\
\noalign{\smallskip}
(s+1)t\Pe\oplus2s(t+1)\Pl&\mbox{for $s$ even $> t$ odd;}\\
\noalign{\smallskip}
s(t+1)\Pe\oplus2s(t+1)\Pl&\mbox{for $s$ even $< t$ odd;}\\
\end{cases}
\end{align}

\

\subsection{Tensor products  between \texorpdfstring{$(k,k)$}{(k,k)}-types modules and syzygies}\label{subsec: kk ot omega}

The following isomorphisms hold for all $s\in\Z_{\geq0}$, $k\in\N$ and $\bt\in\mathfrak{A}$.

\

\begin{align}\label{iso:MLt ot Omega L}
M_k(\Ls,\bt)\ot\Omega^s(\Ls)
\simeq&\,\Omega^s(\Ls)\ot M_k(\Ls,\bt)
\\
\notag\simeq
&
\begin{cases}
M_k(\Le,\bt)\oplus (3s+2)k\Pl&\mbox{for $s$ even;}\\
\noalign{\smallskip}
M_k(\Ls,\bt)\oplus sk\Pe\oplus2sk\Pl&\mbox{for $s$ odd;}
\end{cases}
\\
\noalign{\smallskip}
\label{iso:MLt ot Omega e}
M_k(\Ls,\bt)\ot\Omega^s(\Le)
\simeq&\,\Omega^s(\Le)\ot M_k(\Ls,\bt)\\
\notag\simeq
&
\begin{cases}
M_k(\Ls,\bt)\oplus sk\Pe\oplus2sk\Pl &\mbox{for $s$ even;}\\
\noalign{\smallskip}
M_k(\Le,\bt)\oplus (3s+2)k\Pl&\mbox{for $s$ odd ;}
\end{cases}
\\
\noalign{\smallskip}
\label{iso:Met ot Omega L}
M_k(\Le,\bt)\ot\Omega^s(\Ls)
\simeq&\,\Omega^s(\Ls)\ot M_k(\Le,\bt)\\
\notag\simeq
&
\begin{cases}
M_k(\Ls,\bt)\oplus sk\Pe\oplus2(s+1)k\Pl &\mbox{for $s$ even;}\\
\noalign{\smallskip}
M_k(\Le,\bt)\oplus 3sk\Pl&\mbox{for $s$ odd ;}
\end{cases}
\\
\noalign{\smallskip}
\label{iso:Met ot Omega e}
M_k(\Le,\bt)\ot\Omega^s(\Le)
\simeq&\,\Omega^s(\Le)\ot M_k(\Le,\bt)\\
\notag\simeq
&
\begin{cases}
M_k(\Le,\bt)\oplus3sk\Pl  &\mbox{for $s$ even;}\\
\noalign{\smallskip}
M_k(\Le,\bt)\oplus sk\Pe\oplus2(s+1)k\Pl &\mbox{for $s$ odd ;}
\end{cases}
\end{align}

We will prove \eqref{iso:MLt ot Omega L} and \eqref{iso:Met ot Omega L}. The proof of \eqref{iso:MLt ot Omega e} and \eqref{iso:Met ot Omega e} are analogous. We begin by computing the indecomposable summands of $M_k(\Ls,\bt)\ot\Ls$ and $M_k(\Le,\bt)\ot\Ls$ by induction on $k$; recall that $\Omega^0(\Ls)=\Ls$.  For $k=1$, we claim that the submodule $N=\cA\cdot w\subset M_1(\Ls,\bt)\ot\Ls$, where
\begin{align*}
w=\ow_e\ot v_{(132)}-\frac{t_{13}}{a+2}\ow_{(23)}\ot v_{(12)}+\frac{t_{12}}{a-1}\ow_{(12)}\ot v_{(13)},
\end{align*}
is isomorphic to $M_1(\Le,\bt)$. Indeed, the morphism $f:\Pl\rightarrow M_1(\Ls,\bt)\ot\Ls$, induced by $\delta_\Ls\mapsto w$, satisfies that $\dim\im(f)=6$ and $\xLt\delta_{\Ls}\in\Ker(f)$; we verify this throught a computation in GAP. Hence $N\simeq M_1(\Le,\bt)$ as a consequence of Proposition \ref{prop:projective of MLt and Met}. On the other hand, by applying $-\ot\Ls$ to \eqref{eq:kLs M kLe}, we deduce that $2\Pl$ is a direct summand of $M_1(\Ls,\bt)\ot\Ls$. Since $\soc(2\Pl)=2\Ls$ and $\soc(N)=\Le$, we conclude that 
\begin{align*}
M_1(\Ls,\bt)\ot\Ls\simeq M_1(\Le,\bt)\oplus2\Pl
\end{align*}
by a dimensional argument.
Now, if we apply $-\ot\Ls$ to this isomorphism, we obtain
\begin{align*}
&\left(M_1(\Ls,\bt)\ot\Ls\right)\ot\Ls\simeq  M_1(\Le,\bt)\ot\Ls\oplus2\Pl\ot\Ls\overset{\eqref{iso:proj ot L}}{\simeq}  M_1(\Le,\bt)\ot\Ls\oplus8\Pl\oplus2\Pe\\
&\simeq M_1(\Ls,\bt)\ot\left(\Ls\ot\Ls\right)\overset{\eqref{iso:L ot L}}{\simeq} M_1(\Ls,\bt)\oplus 2M_1(\Ls,\bt)\ot\Pl\overset{\eqref{iso:proj ot L}}{\simeq} M_1(\Ls,\bt)\oplus 10\Pl\oplus2\Pe
\end{align*}
and hence the Krull--Schmidt theorem implies that
\begin{align*}
M_1(\Le,\bt)\ot\Ls\simeq M_1(\Ls,\bt)\oplus2\Pl.
\end{align*}
We have proved \eqref{iso:MLt ot Omega L} and \eqref{iso:Met ot Omega L} for $k=1$ and $s=0$.

Let us continue with the inductive step. We assume that the decomposition of $M_k(\Ls,\bt)\ot\Ls$ and $M_k(\Le,\bt)\ot\Ls$ in \eqref{iso:MLt ot Omega L} and \eqref{iso:Met ot Omega L} hold for $k$. Then, tensoring \eqref{eq:M1 Mk+1 Mk Ls} with $\Ls$, we get the exact sequence
\begin{align*}
0\longrightarrow M_1(\Le,\bt)\oplus2\Pl\longrightarrow M_{k+1}(\Ls,\bt)\ot\Ls\longrightarrow M_k(\Le,\bt)\oplus2k\Pl\longrightarrow0.
\end{align*}
Hence $M_{k+1}(\Ls,\bt)\ot\Ls\simeq N\oplus2(k+1)\Pl$ with $N$ fitting in the exact sequence
\begin{align*}
0\longrightarrow M_1(\Le,\bt)\longrightarrow N\longrightarrow M_k(\Le,\bt)\longrightarrow0.
\end{align*}
By Proposition \ref{prop:M k Le bt}, $N\simeq M_{k+1}(\Le,\bt)$ or $N\simeq M_{1}(\Le,\bt)\oplus M_{k}(\Le,\bt)$. Suppose that the second isomorphism holds. We compute $M_{k+1}(\Ls,\bt)\ot\Ls\ot\Ls$ in two ways:
\begin{align*}
\left(M_{k+1}(\Ls,\bt)\ot\Ls\right)\ot\Ls&\overset{\mbox{\footnotesize{IH}}}{\simeq}M_1(\Ls,\bt)\oplus M_k(\Ls,\bt)\oplus P_1\\
&\simeq M_{k+1}(\Ls,\bt)\ot\left(\Ls \ot\Ls\right)\overset{\eqref{iso:L ot L}}{\simeq} M_{k+1}(\Ls,\bt)\oplus P_2;
\end{align*}
here $P_1$ and $P_2$ denote certain projective modules. By the Krull--Schmidt theorem, the above isomorphisms can not be possible. Therefore $N\simeq M_{k+1}(\Le,\bt)$ and hence $M_{k+1}(\Ls,\bt)\ot\Ls\simeq M_{k+1}(\Le,\bt)\oplus2(k+1)\Pl$. Now, by applying $-\ot\Ls$ to this isomorphism, we can deduce that $M_{k+1}(\Le,\bt)\ot\Ls\simeq M_{k+1}(\Ls,\bt)\oplus2(k+1)\Pl$ by arguing as in the case $k=1$. This complete the proof of the inductive step.

Summarizing, we have computed the indecomposable summands of the tensor products $M_k(\Ls,\bt)\ot\Omega^s(\Ls)$ and $M_k(\Le,\bt)\ot\Omega^s(\Ls)$ for all $k\in\N$ and $s=0$. The indecomposable summands for $s>0$ can be found by induction on $s$, in a similar way to the proof of \eqref{iso:Omega L ot Omega L}. We leave it for the reader.

In order to finish the proof of \eqref{iso:MLt ot Omega L} and \eqref{iso:Met ot Omega L} we must calculate the indecomposable summand of $\Omega^s(\Ls)\ot M_k(\Ls,\bt)$ and $\Omega^s(\Ls)\ot M_k(\Le,\bt)$. This can be made as we have proceed above. We leave it for the reader. We only observe that
\begin{align*}
v_{(132)}\ot\ow_e-\frac{t_{13}}{a+2}v_{(12)}\ot\ow_{(13)}+\frac{t_{12}}{a-1}v_{(13)}\ot\ow_{(23)}
\end{align*}
generates a direct summand of $\Omega^0(\Ls)\ot M_1(\Ls,\bt)$ isomorphic to $M_1(\Le,\bt)$.

\subsection{Tensor products  between \texorpdfstring{$(k,k)$}{(k,k)}-types modules}\label{subsec: kk ot kk}

The following isomorphisms hold for all $k,j\in\N$ and $\bt,\bs\in\mathfrak{A}$. Let us set $i:=\min\{j,k\}$.
\begin{align}\label{iso:MLt ot MLs}
M_k(\Ls,\bt)\ot M_j(\Ls,\bs)
\simeq&\,M_j(\Ls,\bs)\ot M_k(\Ls,\bt)\\
\notag\simeq
&
\begin{cases}
2jk\Pl\oplus jk\Pe&\mbox{if $\bt\not\sim\bs$;}\\
\noalign{\smallskip}
M_{i}(\Le,\bt)\oplus M_{i}(\Ls,\bt)\oplus2jk\Pl\oplus (jk-i)\Pe
&\mbox{if $\bt\sim\bs$;}
\end{cases}
\\
\noalign{\smallskip}
\label{iso:Met ot MLs}
M_k(\Ls,\bs)\ot M_j(\Le,\bt)
\simeq&\,M_j(\Le,\bt)\ot M_k(\Ls,\bs)\\
\notag\simeq
&
\begin{cases}
3jk\Pl&\mbox{if $\bt\not\sim\bs$;}\\
\noalign{\smallskip}
M_{i}(\Le,\bt)\oplus M_{i}(\Ls,\bt)\oplus(3jk-i)\Pl&\mbox{if $\bt\sim\bs$;}
\end{cases}
\\
\noalign{\smallskip}
\label{iso:Met ot Mes}
M_k(\Le,\bt)\ot M_j(\Le,\bs)
\simeq&\,M_j(\Le,\bs)\ot M_k(\Le,\bt)\\
\notag\simeq
&
\begin{cases}
2jk\Pl\oplus jk\Pe&\mbox{if $\bt\not\sim\bs$;}\\
\noalign{\smallskip}
M_i(\Le,\bt)\oplus M_i(\Ls,\bt)\oplus2jk\Pl\oplus (jk-i)\Pe
&\mbox{if $\bt\sim\bs$;}
\end{cases}
\end{align}

\subsubsection{Proof for the case $\bt\not\sim\bs$}\label{subsubsec: MLt ot MLs} We demonstrate first \eqref{iso:MLt ot MLs} by double induction. For $k=1=j$, we see that $M_1(\Ls,\bt)\ot M_1(\Ls,\bs)$ has $2\Pl$ as a direct summand, by tensoring \eqref{eq:kLs M kLe} with $M_1(\Ls,\bs)$ and using \eqref{iso:MLt ot Omega L} for $s=0$. On the other hand, a computation in GAP allows us to verify that 
\begin{align*}
\xtop\cdot(\ow_e\ot\ow_e)&=(\tij{(12)}\sij{(13)}-\tij{(13)}\sij{(12)})\biggl(\ow_{(123)}\ot\ow_{(132)}
-\ow_{(132)}\ot\ow_{(123)}\biggr.\\
&\biggl.
(2a+1)\ow_{(23)}\ot\ow_{(23)}
-(a+2)\ow_{(12)}\ot\ow_{(12)}
-(a-1)\ow_{(13)}\ot\ow_{(13)}
\biggr)
\end{align*}
where $\bt=(\tij{(23)},\tij{(12)},\tij{(13)})$ and $\bs=(\sij{(23)},\sij{(12)},\sij{(13)})$. Thus $\xtop\cdot(\ow_e\ot\ow_e)\neq0$ as $\bt\not\sim\bs$. Then $\Pe$ is a direct summand of $M_1(\Ls,\bt)\ot M_1(\Ls,\bs)$  by Lemma \ref{le:proje direct summand}. Putting all together and by a dimensional argument, we get \eqref{iso:MLt ot MLs} for $k=1=j$. Explicitly,
\begin{align*}
M_1(\Ls,\bt)\ot M_1(\Ls,\bs)\simeq2\Pl\oplus\Pe.
\end{align*}

We continue by proving \eqref{iso:MLt ot MLs} for $k+1$ and $j=1$ assuming that it holds for $k$ and $j=1$. Tensoring the exact sequence \eqref{eq:M1 Mk+1 Mk Ls} by $M_1(\Ls,\bs)$, we get
\begin{align*}
0\longrightarrow2\Pl\oplus\Pe\longrightarrow M_{k+1}(\Ls,\bt)\ot M_1(\Ls,\bs)\longrightarrow 2k\Pl\oplus k\Pe\longrightarrow0
\end{align*}
thanks to the inductive hypothesis. Therefore the middle term must to obey \eqref{iso:MLt ot MLs}.

Finally, we prove \eqref{iso:MLt ot MLs} for $k$ and $j+1$ assuming that it holds for $k$ and $j$. We apply $M_k(\Ls,\bt)\ot-$ to the exact sequence \eqref{eq:M1 Mk+1 Mk Ls}, with $j$ instead of $k$, and obtain
\begin{align*}
0\longrightarrow2k\Pl\oplus k\Pe\longrightarrow M_{k}(\Ls,\bt)\ot M_{j+1}(\Ls,\bs)\longrightarrow 2kj\Pl\oplus kj\Pe\longrightarrow0.
\end{align*}
As above, the middle term must to obey \eqref{iso:MLt ot MLs}. This complete the double induction proof for \eqref{iso:MLt ot MLs}; we notice that $M_k(\Ls,\bt)\ot M_j(\Ls,\bs)\simeq M_j(\Ls,\bs)\ot M_k(\Ls,\bt)$ because our arguments hold for all $k,j\in\N$ and $\bt,\bs\in\mathfrak{A}$.

Now, we demonstrate \eqref{iso:Met ot MLs} for $\bt\not\sim\bs$. We apply $-\ot M_k(\Ls,\bs)$ to the exact sequence \eqref{eq:Mke P MkL} and get
\begin{align*}
0\longrightarrow M_j(\Le,\bt)\ot M_k(\Ls,\bs)\longrightarrow 5jk\Pl\oplus jk\Pe\longrightarrow2jk\Pl\oplus jk\Pe\longrightarrow0;
\end{align*}
the middle term is due to Remark \ref{obs:tensoring by proj} and the term on the right-hand side is due to \eqref{iso:MLt ot MLs}. Therefore $M_j(\Le,\bt)\ot M_k(\Ls,\bs)$ decomposes as in \eqref{iso:Met ot MLs} and the same holds for $M_k(\Ls,\bs)\ot M_j(\Le,\bt)$ because the tensor products in Remark \ref{obs:tensoring by proj} and \eqref{iso:MLt ot MLs} are commutative.

To end, we note that the isomorphism \eqref{iso:Met ot Mes} for $\bt\not\sim\bs$ follows by dualizing \eqref{iso:MLt ot MLs}.

\subsubsection{Proof for the case $\bt\sim\bs$} We can assume $\bt=\bs$. We prove first the decomposition of $M_k(\Ls,\bt)\ot M_j(\Le,\bt)$ for $j\geq k$ given in \eqref{iso:Met ot MLs}. 

By the Morita equivalence and \cite[Lemma 3.28]{Ch4}, there exists an exact sequence
\begin{align*}
0\longrightarrow M_k(\Ls,\bt)\longrightarrow N\longrightarrow \Le\longrightarrow 0 
\end{align*}
with $N\simeq\Omega^k(\Ls)$, for $k$ odd, $N\simeq\Omega^k(\Le)$ otherwise. Tensoring by $M_j(\Le,\bt)$, we get
\begin{align*}
0&\longrightarrow M_k(\Ls,\bt)\ot M_j(\Le,\bt)\overset{\iota}{\longrightarrow}\\
&\overset{\iota}{\longrightarrow} N\ot M_j(\Le,\bt)\overset{\mbox{\tiny\eqref{iso:Met ot Omega L}\eqref{iso:Met ot Omega e}}}{\simeq} M_j(\Le,\bt)\oplus3jk\Pl\overset{\pi}{\longrightarrow} M_j(\Le,\bt)\longrightarrow 0.
\end{align*}
Let $\overline{N}$ be a submodule of $M_k(\Ls,\bt)\ot M_j(\Le,\bt)$ isomorphic to $M_j(\Le,\bt)$ which exists by Lemma \ref{le:Met submod MLt ot Met}, see below.  As $\iota$ is a monomorphism, $\iota(\overline{N})\simeq \overline{N}\simeq M_j(\Le,\bt)$. By observing the socles, we conclude that $\iota(\overline{N})\cap (3jk\Pl)=0$ and hence the above exact sequence looks as follows
\begin{align*}
0&\longrightarrow M_k(\Ls,\bt)\ot M_j(\Le,\bt)\overset{\iota}{\longrightarrow}
\iota(\overline{N})\oplus3jk\Pl\overset{\pi}{\longrightarrow} M_j(\Le,\bt)\longrightarrow 0.
\end{align*}
Since $\pi\circ\iota=0$, the restriction $\pi_{|3jk\Pl}:3jk\Pl\longrightarrow M_j(\Le,\bt)$ is an epimorphism. By Lemma \ref{le:como Chen} and Proposition \ref{prop:projective of MLt and Met}, we have that $\ker(\pi_{|3jk\Pl})\simeq M_k(\Ls,\bt)\oplus(3j-1)k\Pl$. Therefore 
\begin{align*}
M_k(\Ls,\bt)\ot M_j(\Le,\bt)\simeq M_j(\Le,\bt)\oplus M_j(\Le,\bt)\oplus(3j-1)k\Pl
\end{align*}
as we wanted.

The proof for $k\geq j$ is analogous but it starts with an exact sequence involving $M_j(\Le,\bt)$ to which we apply $M_k(\Ls,\bt)\ot-$. 

Now, we can deduce \eqref{iso:Met ot Mes} from \eqref{iso:Met ot MLs} by applying by $\Ls\ot-$. In fact, by the formulas of the above subsections we have that
\begin{align*}
L\ot\biggl(M_k(\Ls,\bt)\ot M_j(\Le,\bt)\biggr)\simeq M_{i}(\Le,\bt)&\oplus M_{i}(\Ls,\bt)
\oplus12jk\Pl\oplus(3jk-i)\Pe
\end{align*}
and, on the other hand, we get
\begin{align*}
\biggl(L\ot M_k(\Ls,\bt)\biggr)\ot M_j(\Le,\bt)\simeq M_k(\Le,\bt)\ot M_j(\Le,\bt)
&\oplus10jk\Pl\oplus2jk\Pe.
\end{align*}
By the associative property, the right hand side of both isomorphisms are isomorphic and hence  \eqref{iso:Met ot Mes} holds due to the Krull--Schmidt theorem; notice that $k$ and $j$ do not play any role in our reasoning and then $M_j(\Le,\bt)\ot M_k(\Le,\bt)\simeq M_k(\Le,\bt)\ot M_j(\Le,\bt)$.

The isomorphism $M_j(\Le,\bt)\ot M_k(\Ls,\bt)$ given in \eqref{iso:Met ot MLs} follows by applying $-\ot\Ls$ to \eqref{iso:Met ot Mes} and arguing as in the above paragraph. 

Finally, the isomorphism  \eqref{iso:MLt ot MLs} for $\bt\sim\bs$ follows by dualizing \eqref{iso:Met ot Mes}.

\begin{lema}\label{le:Met submod MLt ot Met}
If $l\leq j,k$, then $M_k(\Ls,\bt)\ot M_j(\Le,\bt)$ has a submodule isomorphic to $M_l(\Le,\bt)$.
\end{lema}

\begin{proof}
We introduce the following elements in $M_k(\Ls,\bt)\ot M_j(\Le,\bt)$:
\begin{align*}
\fka_{i,\ell}&:=\ow_e^i\ot w_e^\ell
+\ow_{(23)}^i\ot w_{(23)}^\ell
+\ow_{(12)}^i\ot w_{(12)}^\ell
+\ow_{(13)}^i\ot w_{(13)}^\ell\\
&\qquad\qquad\qquad\qquad-\ow_{(123)}^i\ot w_{(132)}^\ell
-\ow_{(132)}^i\ot w_{(123)}^\ell,\\
\fkb_{i,\ell}&:=\ow_e^i\ot w_{(132)}^\ell
-\frac{t_{(13)}}{a+2}\ow_{(23)}^i\ot w_{(12)}^\ell-\frac{t_{(12)}}{a-1}\ow_{(12)}^i\ot w_{(13)}^\ell,\\
\fkc_{i,\ell}&:=-\frac{\widehat{t}_{(13)}}{a+2}\ow_{(23)}^i\ot w_{(12)}^\ell-\frac{\widehat{t}_{(12)}}{a-1}\ow_{(12)}^i\ot w_{(13)}^\ell,
\end{align*}
for all $1\leq i,\ell\leq j, k$. By a computation in GAP\footnote{We compute in GAP the action of $x_\sigma$ on $\fka_{1,2}$, $\fka_{2,3}$ and $\fka_{2,1}$ in $M_3(\Ls,\bt)\ot M_3(\Le,\bt)$. This is enough to deduce the action for all $i$ and $\ell$ because the action of $x_\sigma$ on $\ow^\ell_g$ and $w^{\ell}_g$ only depends on the elements $\ow^\ell_g$, $\ow^{\ell-1}_g$, $w^\ell_g$ and $w^{\ell-1}_g$. We make the same for $\fkb_{i,\ell}$ and $\fkc_{i,\ell}$.}, we see that
\begin{align*}
x_\sigma\fka_{i,\ell}=\widehat{t}\left(\ow_\sigma^{i-1}\ot w_e^\ell-\ow_e^{i}\ot w_\sigma^{\ell-1}\right);
\end{align*}
\begin{align*}
\xij{12}\xij{23}\fkb_{i,\ell}=&
t_{(12)}\fka_{i,\ell}
+\widehat{t}_{(12)}\left(\ow_e^i\ot w_e^{\ell-1}+\ow_{(12)}^{i-1}\ot w_{(12)}^\ell\right)\\
&\qquad\qquad\qquad-\widehat{t}_{(23)}\left(\ow_{(23)}^{i-1}\ot w_{(23)}^\ell-\ow_{(123)}^{i-1}\ot w_{(132)}^\ell\right);\\
\xij{13}\xij{12}\fkb_{i,\ell}=&
t_{(13)}\fka_{i,\ell}
+\widehat{t}_{(13)}\left(\ow_e^i\ot w_e^{\ell-1}+\ow_{(13)}^{i-1}\ot w_{(13)}^\ell\right)\\
&\qquad\qquad\qquad-\widehat{t}_{(12)}\left(\ow_{(12)}^{i-1}\ot w_{(12)}^\ell-\ow_{(123)}^{i-1}\ot w_{(132)}^\ell\right);\\
\xij{12}\xij{23}\fkc_{i,\ell}=&\widehat{t}_{(12)}\left(\ow_{(13)}^i\ot w_{(13)}^\ell-\ow_{(132)}^i\ot w_{(123)}^\ell\right)\\
&\qquad\qquad\qquad-\widehat{t}_{(13)}\left(\ow_{(23)}^i\ot w_{(23)}^\ell-\ow_{(123)}^i\ot w_{(132)}^\ell\right);\\
\xij{13}\xij{12}\fkc_{i,\ell}=&\widehat{t}_{(13)}\left(\ow_{(23)}^i\ot w_{(23)}^\ell-\ow_{(132)}^i\ot w_{(123)}^\ell\right)\\
&\qquad\qquad\qquad-\widehat{t}_{(23)}\left(\ow_{(12)}^i\ot w_{(12)}^\ell-\ow_{(123)}^i\ot w_{(132)}^\ell\right);
\end{align*}
for all $1\leq i,\ell\leq j, k$ and transposition $\sigma$; the summands corresponding to $i-1=0=\ell-1$ must be obviated.

Finally, for $\ell\leq l\leq j,k$, we set
\begin{align*}
\fkw^\ell_{e}:=\sum_{i=1}^\ell\fka_{i,\ell+1-i},\quad &\fkw^\ell_{(132)}:=\sum_{i=1}^\ell\fkb_{i,\ell+1-i}+\sum_{i=1}^{\ell-1}\fkc_{i,\ell-i},\\
\noalign{\smallskip}
\fkw^\ell_{(23)}:=\xij{13}\fkw^\ell_{(132)},\quad
\fkw^\ell_{(12)}&:=\xij{23}\fkw^\ell_{(132)},\quad
\fkw^\ell_{(13)}:=\xij{12}\fkw^\ell_{(132)},\\
\noalign{\smallskip}
\fkw^\ell_{(123)}:=\frac{1}{1+2a}&\xij{12}\xij{13}\fkw^\ell_{(132)}.
\end{align*}
Using the above formulas and the defining relations of $\cA$, we can see that the action of $x_\sigma$ on the elements $\fkw_g^\ell$ obeys the recipe \eqref{eq:M k Le bt}. Therefore $\{\fkw_g^\ell\mid g\in\Sn_3,\,1\leq\ell\leq l\}$ spans a submodule isomorphic to $M_l(\Le,\bt)$.
\end{proof}

\subsection{Tensor products between cosyzygy modules} These tensor products can be calculated by dualizing the isomorphism of Subsection \ref{subsec:tensor syzygy}.

\subsection{Tensor products  between \texorpdfstring{$(k,k)$}{(k,k)}-types modules and cosyzygies}\label{subsec: kk ot omega -}

These tensor products can be calculated by dualizing the isomorphisms of Subsection \ref{subsec: kk ot omega}.

\subsection{The Green ring of \texorpdfstring{$\cA$}{A}}\label{subsec:greenA}

By definition, this ring is generated by the isomorphism classes of modules with operations $[M]+[N]=[M\oplus N]$ and $[M]\cdot[N]=[M\ot N]$, where $[M]$ and $[N]$ denote the respective isomorphism classes of the modules $M$ and $N$. For $k\in\N$ and $\bt\in\mathfrak{A}/\sim$, we recall that
\begin{align*}
1:=[\Le],&&\lambda:=[\Ls],&&\rho:=[\Pl],&&\omega:=[\Omega(\Le)],&&\oomega:=[\Omega^{-1}(\Le)],&&
\mu_{k,\bt}:=[M_k(\Le,\bt)].
\end{align*}

\begin{proof}[Proof of Theorem \ref{teo:green ring}]
In the previous subsections we have seen that the tensor products between indecomposable modules satisfy $M\ot N\simeq N\ot M$. Then the Green ring of $\cA$ is commutative. Let $\Z[\lambda,\rho]$ denote the $\Z$-subalgebra generated by $\lambda$ and $\rho$. The following facts are also deduced from those subsections:

\begin{align*}
[\Pe]&=\lambda\rho-4\rho\quad\mbox{by \eqref{iso:proj ot L};}
\\
[\Omega^s(\Le)]&\in\omega^s\oplus_{i=0}^{s-1}\Z[\lambda,\rho]\omega^i\quad\mbox{by \eqref{iso:Omega e ot Omega e} and induction in $s$;}
\\
[\Omega^s(\Ls)]&\in\lambda\omega^s\oplus_{i=0}^{s-1}\Z[\lambda,\rho]\omega^i\quad\mbox{by \eqref{iso:Omega L ot Omega e} and the above one;}
\\
[\Omega^{-s}(\Le)]&\in\oomega^s\oplus_{i=0}^{s-1}\Z[\lambda,\rho]\oomega^i\quad\mbox{by dual arguments to the above ones;}
\\
[\Omega^{-s}(\Ls)]&\in\lambda\oomega^s\oplus_{i=0}^{s-1}\Z[\lambda,\rho]\oomega^i\quad\mbox{by dual arguments to the above ones;}
\\
[M_k(\Ls,\bt)]&=-2k\rho+\lambda\mu_{k,\bt}\quad\mbox{by \eqref{iso:Met ot Omega L};}
\end{align*}
for all $s,k\in\N$ and $\bt\in\mathfrak{A}$. These imply that $\mathcal{B}$ forms a $\Z$-basis of the Green ring of $\cA$ since the classes of the indecomposable modules does so. In particular, the elements $\lambda$, $\rho$, $\omega$, $\oomega$, and $\mu_{k,\bt}$ generate the Green ring of $\cA$ as a $\Z$-algebra.

Let us prove that the relations in Table \ref{relaciones} hold. Relations \eqref{ring:L ot L} and \eqref{ring:omega e ot omega -e} follow from \eqref{iso:L ot L} and \eqref{eq:omega e ot omega -e}. By \eqref{iso:proj ot L} and Remark \ref{obs:tensoring by proj}, we know that $[\Pe]=\lambda\rho-4\rho$ and $\rho^2=2[\Pe]+10\rho$. These imply \eqref{ring:proj ot L}; \eqref{ring: omega e ot Pl} and \eqref{ring: omega -e ot Pl} are deduced similarly. \eqref{ring:mu k t ot Pl}--\eqref{ring:mu k t ot mu j t} follow from the isomorphisms in Subsections \ref{subsec: kk ot omega}--\ref{subsec: kk ot kk}.

Now, we have an epimorphism from the commutative $\Z$-algebra $R$ presented by generators and relations as in the statement over the Green ring of $\cA$. In order to show the isomorphism it is enough to prove that $R$ is the $\Z$-span of the corresponding elements of $\mathcal{B}$; notice this set is linearly independent over $\Z$ because it projects over a $\Z$-basis. We leave it for the reader to verify that every element 
$$
\lambda^{n_1}\rho^{n_2}\omega^{n_3}\oomega^{n_4}
\mu_{k_5,\bt_5}^{n_5}\cdots \mu_{k_\ell,\bt_\ell}^{n_\ell}\in R
$$
can be expressed as a linear combination of elements in $\mathcal{B}$ using the relations. 
\end{proof}

\subsection{Semisimplification of  \texorpdfstring{$\Rep\cA$}{Rep(A)}}\label{subsec:Semisimplification}

The element $\chi=\sum_{g\in\Sn_3}\sgn(g)\delta_g$ is an involutory pivot of $\cA$ \cite[Proposition 30]{AV2}. Therefore $\cA$ is a spherical Hopf algebras and we can form the quotient category $\underline{\Rep}\cA$ which is a semisimple tensor category. We refer to \cite{AAGTV} for details on this matter. 

\begin{cor}\label{cor:spherical}
Let $\Gamma$ denote the group $C_2\times\Z$. Then
$\underline{\Rep}\cA$ is monoidally equivalent to the category of $\Gamma$-graded finite-dimensional vector spaces.
\end{cor}

\begin{proof}
The simple objects in $\underline{\Rep}\cA$ are the classes of indecomposable modules with non-zero quantum dimension. The quantum dimension of a module is the trace of the action of the pivot $\chi$. Therefore only the simple and the (co)syzygies modules survive in $\underline{\Rep}\cA$. Let $c$ and $z$ be generators of $C_2$ and $\Z$, respectively. The functor $F:\underline{\Rep}\cA\rightarrow\operatorname{vect}^\Gamma$ induced by $F(L)=\ku_c$, $F(\Omega(\Le))=\ku_z$, $F(\Omega^{-1}(\Le))=\ku_{z^{-1}}$, $F(\Omega(\Ls))=\ku_{cz}$ and $F(\Omega^{-1}(\Ls))=\ku_{cz^{-1}}$ gives the desired equivalence thanks to Theorem \ref{teo:green ring}.
\end{proof}

\appendix

\section{On the computations in GAP}\label{sec:GAP}
We set as underlying field $\texttt{F}=\mathbb{Q}(\texttt{a},\texttt{t12},\texttt{t13},\texttt{s12},\texttt{s13})$,  the function field over the rationals in $5$ indeterminates. The first indeterminate represents the scalar involved in the definition of $\cA=\cA_{[(1,a,-1,-a)]}$. Thus, the computations hold for any scalar $a$ except for $1,-\frac{1}{2},-2$. The remaining indeterminates stand for the parameters defining the $(k,k)$-types  modules.

Let $M$ be an $\cA$-module of dimension $m$ with matrix representation $\varrho_M:\cA\rightarrow\Mat_{m\times m}(\texttt{F})$. In order to calculate the action of $x\in\cA$ over $m\in M$ in GAP, we construct the matrices $\varrho_M(\xij{ij})$ and $\varrho_M(\delta_g)$ of all the generators of $\cA$ in a suitable basis of $M$. Then the image of $\varrho_M$ is constructed in GAP using the command {\texttt AlgebraWithOne} and the the above matrices. The elements of $M$ are represented by $m$-uples and the command {\texttt LeftAlgebraModule} allows us to calculate $x\cdot m$. We next explain how we have constructed the matrix representation of the generators.

\subsection{Matrix representation of \texorpdfstring{$\Ls$}{L}} We use the basis and formula in \eqref{eq:action of A}.

\subsection{Matrix representation of tensor products} We use the formulas of the comultiplication given in Section \ref{sec:A} since $\varrho_{M\ot N}(x)=\varrho_M(x\_{1})\ot\varrho_N({x_{(2)}})$. We recall that the command {\texttt KroneckerProduct} builds up the tensor product of vector spaces.

\subsection{Matrix representation of \texorpdfstring{$\Pl$}{PL}} By Subsection \ref{subsec: L ot L}, we know that $\Pl$ is isomorphic to the submodule generated by $n_1\in\Ls\ot\Ls$. Thus, the command {\texttt SubAlgebraModule} allows us work with $\Pl$ as a submodule of $\Ls\ot\Ls$. In particular, we can consider the basis $\{x\cdot n_1\mid x\in\B\}$ of $\Pl$. Then, we construct the matrix representation of the generators using the command {\texttt MatrixOfAction}. Of course, one can construct these matrices by hand using the formulas \cite[$(19)-(52)$]{AV2} but it is more tedious.

\subsection{Matrix representation of \texorpdfstring{$\Pe$}{Pe}} 

We proceed as for $\Pl$. We use that $\Pe$ is a submodule of $\Ls\ot\Pl$ by Subsection \ref{subsec: ot projective}. A computation in GAP shows that $v_{(123)}\ot\delta_L\in\Ls\ot\Pl$ generates such a submodule.

\subsection{Matrix representation of the \texorpdfstring{$(k,k)$}{(k,k)}-types  modules} 
We use the bases and formulas in \eqref{eq:M k Le bt} and \eqref{eq:M k Ls bt}. The elements $(t_{(23)},t_{(12)},t_{(13)}),(\widehat{t}_{(23)},\widehat{t}_{(12)},\widehat{t}_{(13)})\in\mathfrak{A}$ are represented by $(-\texttt{t12}-\texttt{t13},\texttt{t12}, \texttt{t13})$ and $(-\texttt{s12}-\texttt{s13},\texttt{s12}, \texttt{s13})$, respectively. For our purposes, it is enough to make calculations for $k\in\{1,2,3\}$ because the action of $\xij{ij}$ on the basis elements $\ow^\ell_g$ and $w^{\ell}_g$ only depends on the elements $\ow^\ell_g$, $\ow^{\ell-1}_g$, $w^\ell_g$ and $w^{\ell-1}_g$.

In Subsection \ref{subsubsec: MLt ot MLs}, we carry out computations in $M_1(\Ls,\bt)\ot M_1(\Ls,\bs)$ with $\bt\not\sim\bs$. Then we use $(-\texttt{t12}-\texttt{t13},\texttt{t12}, \texttt{t13})$ in the representation of $M_1(L,\bt)$ and $(-\texttt{s12}-\texttt{s13},\texttt{s12}, \texttt{s13})$ in the representation of $M_1(L,\bs)$ since, for $k=1$, the element $(\widehat{t}_{(23)},\widehat{t}_{(12)},\widehat{t}_{(13)})$ does not play any role in the definitions of these modules.

\end{document}